\documentclass[12pt,final]{article}
\usepackage{amsmath,amsthm,amsfonts,amssymb,graphicx,hyperref,enumerate,psfrag}
\usepackage{fullpage}

\usepackage{xcolor}
\usepackage[utf8]{inputenc}
\usepackage{dsfont,mathtools,amssymb}
 
\newtheorem{lemma}{Lemma}

\newtheorem{claim}{Claim}
\newtheorem{remark}{Remark}
\newtheorem{proposition}{Proposition}
\newtheorem{theorem}{Theorem}
\newtheorem{conjecture}[lemma]{Conjecture}

\newtheorem*{ackno}{Acknowledgement}

\newcommand{\EE}{{\mathbf{E}}}

\newcommand{\bbS}{{\mathbb{S}}}

\newcommand{\dd}{\,{\rm d}}
\newcommand{\dist}{\,{\rm dist}}
\newcommand{\PP}{\mathbf{P}}

\newcommand{\Lip}{\mathrm{Lip}}
\newcommand{\Z}{\mathbb {Z}}
\newcommand{\N}{\mathbb {N}}
\newcommand{\R}{\mathbb {R}}

\newcommand{\bA}{A^c}
\newcommand{\cF}{\mathcal {F}}

\newcommand{\cB}{\mathcal {B}}

\newcommand{\cN}{\mathcal {N}}

\newcommand{\cE}{\mathcal {E}}
\newcommand{\cG}{\mathcal {G}}

\newcommand{\cP}{\mathcal {P}}

\newcommand{\bS}{{\bf S}}

\newcommand{\cH}{\mathcal{H}}

\newcommand{\ent}{{\mathrm{Ent}}}
\newcommand{\var}{{\mathrm{Var}}}
\newcommand{\cov}{{\mathrm{Cov}}}
\title{Entropy factorization via  curvature}
\author{Pietro Caputo and Justin Salez}

\begin{document}

\maketitle 
\begin{abstract}
We develop a new framework for establishing approximate factorization of entropy on arbitrary probability spaces, using a geometric notion known as non-negative sectional curvature. The resulting estimates are equivalent to entropy subadditivity and generalized Brascamp-Lieb inequalities, and provide a sharp modified log-Sobolev inequality for the Gibbs sampler of several particle systems in both continuous and discrete settings. 
The method allows us to obtain simple proofs of known results, as well as some new inequalities. We illustrate this through various applications, including discrete Gaussian free fields on arbitrary networks, the down-up walk on uniform $n$-sets, the uniform measure over permutations, and the uniform measure on the unit sphere in $\R^n$. Our method also yields a simple, 
coupling-based proof of the celebrated logarithmic Sobolev inequality for  Langevin diffusions in a convex potential, which is one of the most emblematic applications of the Bakry-\'Emery criterion. 
\end{abstract}


\tableofcontents

\section{Introduction}
\subsection{Motivation}
In this motivational section, we consider a product probability measure $\PP$ on a $n-$dimensional measurable space $(\Omega,\cF)=\bigotimes_{i=1}^n(\Omega_i,\cF_i)$.
\label{sec:motivation}
\paragraph{Variance tensorization.} We write $\EE[f]$ and $\var(f)$ for the expectation and variance of a random variable $f\in L^2(\Omega,\cF,\PP)$. For  $i\in[n]$, we let $Z_i\colon \Omega\to \Omega_i$ denote the projection onto the $i-$th coordinate, and we define $\EE_i[f]:=\EE[f|Z_1,\ldots,Z_{i-1},Z_{i+1},\ldots,Z_n]$ and $\var_i(f):=\EE_i[f^2]-\EE^2_i[f]$. With this notation, the celebrated Efron-Stein inequality asserts that 
\begin{eqnarray}
\label{eq:VT}
\forall f\in L^2,\qquad \var\left(f\right) & \le & \sum_{i=1}^n\EE\left[\var_i(f)\right].
\end{eqnarray}
This property is the starting point of the well-developed theory of concentration of Lipschitz observables on product spaces, as exposed in the comprehensive textbook \cite{MR3185193} or the beautiful lecture notes \cite{van2014probability}. From a dynamical viewpoint, it can also be regarded as a dimension-free Poincar\'e inequality for the Gibbs sampler associated with $\PP$. 

\paragraph{Entropy tensorization.}
A far-reaching observation  is that a similar relation holds between the entropy  $\ent(f):=\EE[f \log f]-\EE[f]\log\EE[f]$ and its conditional versions $\ent_i(f):=\EE_i[f \log f]-\EE_i[f]\log\EE_i[f]$, $i\in[n]$. Specifically, we have 
\begin{eqnarray}
\label{eq:ET}
\forall f\in L\log L,\qquad \ent\left(f\right) & \le & \sum_{i=1}^n\EE\left[\ent_i(f)\right],
\end{eqnarray} where $L\log L$ denotes the set of measurable functions $f\colon\Omega\to\R_+$ such that $\EE[|f\log f|]<\infty$.  The functional inequality (\ref{eq:ET}) implies (\ref{eq:VT}) by a classical perturbation argument around constant functions. It admits several notable consequences, including a sub-Gaussian concentration estimate for Lipschitz observables, and a dimension-free modified log-Sobolev inequality for the Gibbs sampler associated with $\PP$. We again refer the  reader to  \cite{MR3185193,van2014probability} for details. In fact, this entropy tensorization  happens to be a special case of an even more general inequality due to Shearer, which we refer to as \emph{block factorization} of entropy.

\paragraph{Block factorization.} Given a set $A\subset[n]$,   we write $\EE_A[\cdot]:=\EE[\cdot|Z_j\colon j\in \bA]$ for the conditional expectation given all coordinates in $\bA:=[n]\setminus A$, and $\ent_A[f]:=\EE_A[f\log f]-\EE_A[f]\log\EE_A[f]$ for the corresponding conditional entropy. With this notation at hand, Shearer's inequality (see, e.g., \cite{MR2683430}) asserts that 
\begin{eqnarray}
\label{eq:BT}
\forall f\in L\log L,\qquad \theta_\star\,\ent\left(f\right) & \le & \sum_{A\subset[n]}\theta_A\,\EE\left[\ent_A(f)\right],
\end{eqnarray}
for any non-negative weights $(\theta_A)_{A\subset[n]}$, where $\theta_\star:=\min_{1\le i\le n}\sum_{ A\ni i}\theta_A$ is the smallest marginal weight. This  functional  inequality considerably generalizes the entropy tensorization (\ref{eq:ET}), which corresponds to the choice $\theta_A={\bf 1}_{|A|=1}$. It implies a dimension-free modified log-Sobolev inequality for the \emph{Block Dynamics} on $(\Omega,\cF,\PP)$ that consists in re-sampling any block of coordinates $(Z_i\colon i\in A)$  at rate $\theta_A$ according to its conditional law given $(Z_j\colon j\in \bA)$. 

\paragraph{Approximate block factorization.} The concept of entropy factorization, and its use in the analysis of log-Sobolev inequalities go back to classical work on the Glauber dynamics for non-product measures such as lattice spin systems \cite{Lu-Yau,SZ92a,MR1746301,dai2002entropy,MR3413687}. This point of view was then revisited in \cite{MR2995699,MR3434252,MR4015662}.
Motivated by the powerful consequences of Shearer's inequality in probability, combinatorics, and functional analysis, \cite{MR4334247,MR4415182,BristielC}  have recently investigated more systematically the possibility of establishing  approximate versions of  (\ref{eq:BT}) that apply to non-product measures.  More precisely, given a non-negative vector $\theta=(\theta_A)_{A\subset[n]}$ and an arbitrary probability measure $\PP$ on an $n-$dimensional measurable space $(\Omega,\cF)=\bigotimes_{i=1}^n(\Omega_i,\cF_i)$, one looks for a constant $\kappa=\kappa(\PP,\theta)>0$, as large as possible, such that 
\begin{eqnarray}
\label{eq:ABT}
\forall f\in L\log L,\qquad \kappa \,\ent\left(f\right) & \le & \sum_{A\subset[n]}\theta_A\,\EE\left[\ent_A(f)\right].
\end{eqnarray}
As in the case of Shearer's inequality, one important motivation for studying approximate block factorizations is the fact that the constant $\kappa$ in \eqref{eq:ABT} provides a lower bound on the modified log-Sobolev constant for the block dynamics associated with the weights $\theta$.  We refer the reader to  the recent lecture notes \cite{entropy} for a self-contained introduction to this active line of research. 

\paragraph{Entropy subadditivity and Brascamp-Lieb inequalities.}
Another important motivation stems from the equivalence between block factorizations, entropy subadditivity and the Brascamp-Lieb inequalities.
Indeed, using the classical
chain rule for entropy, for any sub $\sigma$-algebra $\cG\subset \cF$ we write
\begin{eqnarray}
\label{def:chain}
\forall f\in L\log L,\qquad \ent(f) & = & \EE[\ent(f |{\cal G})]+\ent(\EE[f |{\cal G}]),
\end{eqnarray}
where $\EE\left[\ent( f|\cG)\right]:=\EE\left[f\log(f/\EE( f|\cG))\right]$ is the averaged conditional entropy of $f$ with respect to $\cG$. 
Assuming (without  loss of generality) that the weights $(\theta_A)_{A\subset[n]}$ are normalized so as to sum to $1$, we can rewrite the approximate entropy factorization  \eqref{eq:ABT} equivalently as an \emph{entropy subadditivity} statement of the form:
\begin{eqnarray}
\label{standardized}
\forall f\in L\log L,\qquad\sum_{A\subset[n]}\theta_A\,\ent\left(\EE_A[f]\right) & \le & (1-\kappa)\ent(f).
\end{eqnarray}
On the other hand, it is well known  \cite{carlen2009subadditivity}  that by Legendre duality,  the subadditivity \eqref{standardized} is equivalent to the Brascamp-Lieb inequality
\begin{align}\label{eq:BL1}
 \EE\left[
  \textstyle{\prod_{A\subset [n]}} 
 g_A(Z_{\bA})^{c_A}\right] \leq  \textstyle{\prod_{A\subset [n]}}
 \EE\left[g_A(Z_{\bA})\right]^{c_A}\,,
\end{align}
for all bounded measurable functions $g_A:\R^{|\bA|}\mapsto\R_+$, where, for all $A\subset [n]$, $c_A:=\frac{\theta_{A}}{1-\kappa}$, and $Z_A$ denotes the set of variables
$(Z_i, i\in A)$. We note that if $f$ is a probability density with respect to $\PP$, so that $f\PP$ defines a probability measure  on $\Omega$, then $\EE_A[f]$ is the density of the pushforward of $f\PP$ under the projection $Z\mapsto Z_{\bA}$  (the marginal of $f\PP$ on $Z_{\bA}$). In fact, the equivalence 
between \eqref{eq:BL1} and \eqref{eq:ABT} applies more generally when the projection $Z\mapsto Z_{A^c}$ is replaced by 
an arbitrary measurable map $Z\mapsto B_A (Z)$, and $\EE_A[f]$ is replaced by the conditional expectation $\EE[f|B_A(Z)]$; see \cite[Theorem 2.1]{carlen2009subadditivity}. 
The inequalities \eqref{eq:BL1} obtained in this way are referred to as Brascamp-Lieb (B-L) type inequality, in analogy with the classical B-L inequality, which corresponds to the  setting where $\PP$ is the Lebesgue measure on $\R^n$ and the measurable maps $B_A$ are linear; see \cite{brascamp1976best,lieb1990gaussian,barthe1998reverse,bennett2008brascamp,lehec2014short}. 

\paragraph{Our contribution.} 
We adopt 
a more general viewpoint on the inequalities \eqref{eq:BL1} and \eqref{eq:ABT}, by replacing conditional expectations $\EE_A[f]$
with arbitrary, not necessarily self-adjoint, Markov operators $f\mapsto T_A f$. It turns out that it is not hard to extend the duality to the new framework, thus generalizing the equivalence proved in  \cite{carlen2009subadditivity}; see Section \ref{sec:general}. 

Within this general framework,  we propose a new approach towards establishing the approximate entropy factorization  (\ref{eq:ABT}), based on \emph{non-negative sectional curvature}. This simple probabilistic and geometric notion, an $L^\infty$ version of Ollivier's Ricci curvature \cite{ollivier2009ricci},  has recently been shown to play a fundamental role in quantifying the entropy dissipation rate of Markov processes \cite{pedrotti2023contractive,munch2023ollivier,caputo2024entropy}. In fact, our main result, stated in  Section \ref{sec:main} below, is  sufficiently general  to also contain the main finding of \cite{caputo2024entropy}.
We demonstrate the robustness of our method through various applications in both continuous and discrete settings.
For instance, the method enables us to derive optimal entropy factorizations for Gaussian free fields on arbitrary networks, providing optimal bounds on the spectral gap and relative entropy decay of any weighted block dynamics. We also discuss analogous bounds for the uniform measure over permutations and over the unit sphere in $\R^n$. Although obtaining optimal constants for these models with any choice of weights is significantly challenging, our method achieves bounds that match the state-of-the-art.
Additionally, we apply our method to the down-up walk for uniform 
$n$-sets, a special case of the well-studied down-up walk on the bases of a matroid. Recent studies have obtained remarkable entropy contraction bounds for these processes using properties of log-concave polynomials (see \cite{cryan2021modified, anari2022entropic}). In the special case of uniform $n$-sets, our method produces a slightly better bound than that derived from log-concavity.
Finally, we highlight the versatility of our method by providing a straightforward coupling proof of the well known Bakry-\'Emery criterion for Langevin diffusions in a convex potential.

\subsection{General framework}\label{sec:general}
\label{sec:main}
We consider a generalized version of the inequality 
\eqref{standardized} where 
the conditional expectations are replaced with  arbitrary Markov operators. 
\paragraph{Markov operators. }From  now on, we consider an arbitrary  probability space $(\Omega,\cF,\PP)$, and we simply write $L^p$ for $L^p(\Omega,\cF,\PP)$, $p\in[1,\infty]$. We recall that a \emph{transition kernel} on $\Omega$ is  a function $T\colon \Omega\times\cF\to [0,1]$ such that
\begin{enumerate}[(i)]
\item for each  point $x\in\Omega$, the function $E\mapsto T(x,E)$ is a probability measure;
\item for each event $E\in\cF$, the function $x\mapsto T(x,E)$ is measurable. 
\end{enumerate}
Integrating with respect to such a kernel naturally gives rise to a bounded linear operator on  $L^\infty$, which is also denoted by $T$, and whose action is as follows:
\begin{eqnarray}
\label{def:Tf}
\forall x\in\Omega, \qquad (Tf)(x) & := & \int_\Omega f(y)\,T(x,{\rm d} y). 
\end{eqnarray}
Of course, $Tf$ is non-negative as soon as  $f$ is, and   $T{\bf 1}={\bf 1}$: such operators are usually known as \emph{Markov operators}. Note that the underlying transition kernel can be recovered from the operator via the formula $T(x,E)=(T{\bf 1}_E)(x)$, so that the above identification is legit.  We will always assume that $T$ is \emph{measure preserving}, in the sense that for all $f\in L^\infty$,
\begin{eqnarray}
\label{def:statio}
 \EE[Tf] & = & \EE[f],
\end{eqnarray}
where we recall that $\EE[\cdot]$ denotes the expectation w.r.t. $\PP$. Thanks to this stationarity and Jensen's inequality, the formula (\ref{def:Tf}) actually defines a bounded linear operator $T\colon L^p\to L^p$ with operator norm $1$,  for any $p\in[1,\infty]$. Its adjoint $T^\star$ is also a measure-preserving Markov operator, characterized by the relation
\begin{eqnarray}
\forall f,g\in L^2,\qquad \EE[(T^\star f)g] & = & \EE[(T g)f].
\end{eqnarray}
As above, we write $f\in L\log L$ to mean that $f\colon\Omega\to\R_+$ is measurable with $\EE[|f\log f|]<\infty$. Note that $Tf$ is then well-defined and in $L\log L$, with
\begin{eqnarray*}
\ent(Tf) & \le & \ent(f).
\end{eqnarray*}
We are now in position to introduce  the general  inequality that we propose to investigate. 

\paragraph{General functional inequality.}\label{sec:generalf} For an integer $M\ge 1$, we assume to be given a  family $(T_1,\ldots,T_M)$ of measure-preserving transition kernels  on our workspace $(\Omega,\cF,\PP)$, identified with the corresponding Markov operators, along with a probability vector $(\theta_1,\ldots,\theta_M)$. Our aim is to find a constant $\kappa\in[0,1]$, as large as possible, such that 
\begin{eqnarray}
\label{eq:main}
\forall f\in L\log L,\qquad \sum_{i=1}^M \theta_i\,\ent(T_i^\star f) & \le & (1-\kappa)\ent(f).
\end{eqnarray}
Since, for any $\sigma-$field $\cG\subset\cF$, the conditional expectation $f\mapsto\EE[f|\cal G]$ is a measure-preserving self-adjoint Markov operator, the inequality \eqref{eq:main} contains \eqref{standardized} as a special case. Also  note that, by the convexity of the functional $f\mapsto \ent(f)$, the inequality (\ref{eq:main})  implies that the average Markov operator $T:=\sum_{i=1}^M\theta _i T_i$ satisfies the $1$-step entropy contraction property
\begin{eqnarray}
\forall f\in L\log L,\qquad \ent(T^\star f) & \le & (1-\kappa)\ent(f),
\end{eqnarray}
which we can iterate to deduce exponential convergence to equilibrium, in relative entropy, at rate $\kappa$, for the  Markov chain with transition kernel $T$. In light of those two observations, it seems of considerable interest to find  simple and broadly applicable criteria that guarantee  (\ref{eq:main}) with an explicit (and, ideally, sharp) constant $\kappa$. This is precisely the aim of the present paper.


\paragraph{Duality and generalized B-L inequalities.}
Before presenting our main result, we observe that the  equivalence 
between \eqref{eq:BL1} and \eqref{eq:ABT} established in \cite{carlen2009subadditivity} admits the following generalization in our setting.   
\begin{theorem}[Duality]
\label{th:dual}
Consider  Markov operators $(T_1,\dots,T_M)$,   a constant $\kappa\in(0,1)$, and a probability vector $(\theta_1,\dots,\theta_M)$,  and define 
\[
c_i = \frac{\theta_i}{1-\kappa}\,,\qquad i=1,\dots,M.
\] 
The following statements are equivalent. 
\begin{enumerate}[(a)]
\item Entropy subadditivity:
\begin{eqnarray}
\label{eq:subaddi}
\forall f\in L\log L,\qquad \sum_{i=1}^M \theta_i\,\ent(T_i^\star f) & \le & (1-\kappa) \ent(f)\,.
\end{eqnarray}
\item Generalized B-L inequality: 
\begin{eqnarray}\label{eq:genBL}
\forall \varphi_1,\dots,\varphi_M\in L^\infty,\qquad 
\EE\left[\prod_{i=1}^Me^{c_i T_i \varphi_i}\right] & \le & \prod_{i=1}^M\EE\left[e^{ \varphi_i}\right]^{c_i} \,.
\end{eqnarray}
\end{enumerate}
\end{theorem} 
We remark that when the Markov operators $T_i$ are conditional expectations of the form $T_if = \EE[f|\cF_i]$ for some sub $\sigma$-algebras $\cF_i$, Theorem \ref{th:dual} yields 
the duality proved in \cite[Theorem 2.1]{carlen2009subadditivity}.
In that case  $T_i$ is an orthogonal projection in $L^2$, and therefore $T_i$ is self-adjoint, and by the chain rule \eqref{def:chain} 
one has that both statements (1) and (2) are equivalent to an entropy factorization statement of the form
\begin{eqnarray}
\label{eq:subaddio}
\forall f\in L\log L,\qquad \sum_{i=1}^M \theta_i\,\EE\left[\ent( f|\cF_i)\right] & \ge & \kappa\,\ent(f)\,.
\end{eqnarray}
We note that the case where $T_i$ is not a conditional expectation goes beyond the common B-L inequality setup.
The proof, however,  boils down to essentially the same argument as in 
\cite{carlen2009subadditivity}, which is based on the well known variational principle for relative entropy stating that
\begin{align}\label{entfvarprin}
\ent (f) \;=\; \sup_{h} \;\left\{\EE[fh] - \log\EE[e^h]\right\} ,
\end{align}
where $h$ ranges over all functions on $\Omega$ such that $e^h\in L^1$. 
\begin{proof}[Proof of Theorem \ref{th:dual}]
Suppose (b) holds. Let $f\in L\log L$ and set $h = \sum_{i=1}^Mc_iT_i\log (T_i^\star f)$. Without loss of generality we may assume that $f\in L^\infty$ and   $\EE(f)=1$. 
From the variational principle \eqref{entfvarprin}, 
one has
\begin{align}\label{equiva1}
\ent f \geq \EE[fh] - \log\EE[e^h] = 
\sum_ic_i\EE\left[f T_i\log (T_i^\star f)\right] -\log\EE\left[\prod_{i}
e^{c_iT_i\varphi_i}
\right]\,,
\end{align}
where we define $\varphi_i := \log (T_i^\star f)$. From the assumption (b) it follows that 
\[
\log\EE\left[\prod_{i}
e^{c_iT_i\varphi_i}
\right]\leq \sum_i c_i \log \EE\left[
e^{\varphi_i}
\right].
\]
Since 
$\EE\left[
e^{\varphi_i}
\right]=\EE\left[T_i^\star f
\right] = 1$, 
using also $\EE\left[f T_i\log (T_i^\star f)\right]=\ent(T_i^\star f)$, \eqref{equiva1} implies
\begin{align}\label{equiva1a}
\ent f \geq
\sum_ic_i\,\ent(T_i^\star f)
\,.
\end{align}
This proves (b) $\implies $ (a). To prove the converse, take arbitrary functions $\varphi_i\in L^\infty$, and set $h = \sum_{i=1}^Mc_iT_i\varphi_i$. 
Taking $f = e^{h}/\EE[e^{h}]$, we observe that 
\begin{align}\label{equiva3}
\ent f =   \EE[fh] - \log \EE[e^h]  = \sum_i
c_i\EE[T_i^\star f \varphi_i]
-\log\EE\left[\prod_{i}e^{c_iT_i\varphi_i}
\right].
\end{align}
From the assumption (a) and \eqref{equiva3} it follows that 
\begin{align}\label{equiva3a}
\log\EE\left[\prod_{i}e^{c_iT_i\varphi_i}
\right]\leq \sum_ic_i\EE[T_i^\star f \varphi_i] - \sum_ic_i\ent (T_i^\star f) \,.
\end{align}
From the variational principle \eqref{entfvarprin}, for each $i\in[n]$ one has
\begin{align}\label{equiva1b}
\ent (T_i^\star f) = \EE\left[T_i^\star f\log T_i^\star f\right]\geq \EE\left[T_i^\star f\varphi_i\right] -\log\EE\left[e^{\varphi_i}
\right].
\end{align}
If we multiply by $c_i$ and sum over $i$ in \eqref{equiva1b}, and then take  exponentials in  \eqref{equiva3a} we obtain
\[
\EE\left[\prod_{i}e^{c_iT_i\varphi_i}
\right]\leq \prod_{i} \EE\left[e^{\varphi_i}
\right]^{c_i}\,.
\]
This proves (a) $\implies $ (b). 
\end{proof}
After completing this work, we learned that the generalized form of duality proposed in Theorem \ref{th:dual} had already been observed; see \cite{ahlswede1976spreading} for an early result in  the case $M=1$, and \cite{liu2017information} for a general result analogous to the one considered here. We also refer  to \cite{liu2017information} for a list of interesting applications of this generalized framework, such as hypercontractivity, strong data processing inequalities and other information-theoretic questions.

\subsection{Main result} We henceforth assume that $\Omega$ is equipped with a  metric $\dist(\cdot,\cdot)$, which we are free to choose as we want as long as the  square-integrability condition 
\begin{eqnarray}
\label{assume:squareint}
\int_{\Omega}\dist^2(o,x)\PP({\rm d}x) & < & +\infty,
\end{eqnarray}
holds for some (and hence every) point $o\in\Omega$. 
 We write $W_1$ and $W_\infty$ for the resulting $L^1$ and $L^\infty-$Wasserstein distances, defined for probability measures $\mu,\nu\in\cP(\Omega)$ as follows:
\begin{eqnarray*}
W_1(\mu,\nu) & := & \inf_{X\sim \mu,Y\sim\nu}{\mathbb E}[\dist(X,Y)],\\
W_\infty(\mu,\nu) & := & \inf_{X\sim \mu,Y\sim\nu}\mathrm{ess\,sup}\dist(X,Y),
\end{eqnarray*}
where the infimum runs over all random pairs $(X,Y)$ with marginals $\mu$ and $\nu$. Also, we define the \emph{essential Lipschitz constant} of a  function $f\colon \Omega\to\R$ as
\begin{eqnarray*}
\Lip(f) & := & \inf_{N\in\cN}\, \sup\left\{\frac{|f(x)-f(y)|}{\dist(x,y)}\colon (x,y )\in (\Omega\setminus N)^2, x\ne y  \right\},
\end{eqnarray*}
where $\cN\subset \cF$ denotes the set of zero-probability events, henceforth  called \emph{null sets}. 
\begin{theorem}[Main result]\label{th:main} Consider Markov operators $(T_1,\dots,T_M)$,   and a probability vector $(\theta_1,\dots,\theta_M)$.  Assume that,  for some numbers $\ell_1,\ldots,\ell_M\ge 0$, and a constant $\kappa\in[0,1]$, 
the following conditions are satisfied:
\begin{enumerate}[(i)]
\item (Sectional curvature). For each $1\le i\le M$ and all $x,y$ outside a null set,
\begin{eqnarray*}
W_\infty\left(T_i^\star(x,\cdot),T_i^\star(y,\cdot)\right)& \le & \ell_i\dist(x,y).
\end{eqnarray*}
\item (Average curvature). For all $x,y$ outside a null set,
\begin{eqnarray}
\label{assume:contractive}
\sum_{i=1}^M\theta_i\ell_i W_1\left(T_i(x,\cdot),T_i(y,\cdot)\right)  & \le & \left(1-\kappa\right)\dist(x,y).
\end{eqnarray}
\item (Regularity). For each $1\le i\le M,$ we have
\begin{eqnarray*}
\forall f\in L^\infty, \qquad \Lip(T_if) & < & \infty.
\end{eqnarray*}
\end{enumerate}
Then, the entropy contraction property (\ref{eq:main}) holds with constant $\kappa$.
\end{theorem} 
There are several important remarks to make about this result. 
\begin{remark}[Regularization]\label{rk:regul}  Assumption (iii)  trivially holds when $\Omega$ is finite and, more generally, when the metric $\dist(\cdot,\cdot)$ is discrete. When it fails, it can often be bypassed  by an appropriate regularization. More precisely, suppose that our workspace $(\Omega,\cF,\PP)$ admits a family  of measure-preserving transition kernels $(P_\tau)_{\tau\in(0,1)}$  such that
\begin{enumerate}
\item (Continuity). For each $f\in L^2$, we have $P_\tau f\xrightarrow[\tau\to 0]{} f$ in  $L^2$.
\item (Regularity). For each $\tau\in(0,1)$, and each $f\in L^\infty$, we have $\Lip(P_\tau f) < \infty$.
\item (Non-negative sectional curvature). For some sequence $\epsilon_\tau\to 0$, as $\tau\to0$, for all $x,y\in\Omega$ outside a null set,
\begin{eqnarray*}
W_\infty\left(P_\tau(x,\cdot),P_\tau(y,\cdot)\right)& \le & (1+\epsilon_\tau)\dist(x,y)\\
W_\infty\left(P_\tau^\star(x,\cdot),P_\tau^\star(y,\cdot)\right)& \le & (1+\epsilon_\tau)\dist(x,y)\,.
\end{eqnarray*}
\end{enumerate}
Then, Assumption (iii) in Theorem \ref{th:main} is unnecessary: indeed, we can apply the theorem to the operators $(P_\tau T_i)_{1\le i\le M}$ -- which do fulfill all the requirements -- instead of $(T_i)_{1\le i\le M}$ and then let $\tau\to 0$ to obtain the desired  conclusion. This regularization trick will be implemented in the Gaussian case in Section \ref{sec:gauss}, using the Ornstein-Uhlenbeck semi-group.  
\end{remark}
\begin{remark}[Peres-Tetali conjecture]The result is already highly non-trivial in the special case where $M=1$, $\ell=1$ and  $\Omega$ is finite: indeed, it is then exactly the main result of the recent preprint \cite{caputo2024entropy}, devoted to the celebrated Peres-Tetali conjecture. The present paper can be seen as a (powerful) refinement of the method used therein.
\end{remark}
\begin{remark}[The product case]\label{rem:product}
Recovering the original Shearer inequality (\ref{eq:BT}) from Theorem \ref{th:main} is easy: consider a product space $(\Omega,\cF,\PP)=\bigotimes_{i=1}^n(\Omega_i,\cF_i,\PP_i)$, for some $n\in\N$, and define a family of self-adjoint Markov operators $(T_A)_{A\subset [n]}$ 
 by the formula
\begin{eqnarray*}
T_A f & := & \EE_A[f] \ = \ \EE[f|Z_i,i\in\bA],
\end{eqnarray*}
where we recall that $Z=(Z_1,\ldots,Z_n)$ denotes the identity map on $\Omega$. 
More explicitly,  $T_A(x,\cdot)$ is the law of the random vector obtained from $x$ by replacing $x_i$ with $Z_i$ for each $i\in A$. Now, equip $\Omega$  with the Hamming distance $\dist(x,y):=\#\{i\in[n]\colon x_i\ne y_i\}$. Then,
\begin{eqnarray*}
 W_\infty\left(T_A(x,\cdot),T_A(y,\cdot)\right) & \le & \#\{i\in\bA\colon x_i\ne y_i\} \ \le \ \dist(x,y),
 \end{eqnarray*} 
as witnessed by the trivial coupling that consists in  replacing both $x_i$ and $y_i$ by $Z_i$  for each $i\in A$. Since $W_1\le W_\infty$, we moreover have, for any probability vector $(\theta_A)_{A\subset[n]}$,
 \begin{eqnarray*}
\sum_{A\subset[n]}\theta_A W_1\left(T_A(x,\cdot),T_A(y,\cdot)\right) & \le & \sum_{i=1}^n {\bf 1}_{x_i\ne y_i}\left(1-\sum_{A\ni i}\theta_A\right)\ \le \ (1-\theta_\star)\dist(x,y).
 \end{eqnarray*} 
Thus, Theorem \ref{th:main} applies with $M=2^n$, $\ell\equiv 1$, $\kappa=\theta_\star$. Therefore, 
for all $f\in L\log L$,
\begin{eqnarray*}
\sum_{A\subset[n]}\theta_A\,\ent\left(\EE_A[f]\right) & \le & (1-\theta_\star)\ent(f).
\end{eqnarray*}
Using the chain rule (\ref{def:chain}), this is equivalent to (\ref{eq:BT}).
\end{remark}

\subsection{Some examples}\label{sec:applications}
We consider several applications of our main result, to both continuous and discrete settings. 
\paragraph{Discrete Gaussian Free Field.}
%

The discrete Gaussian free field is a Gaussian probability measure on $\R^n$, with covariance of the form $\Gamma=({\rm Id} - P)^{-1}$, where $P$ is a symmetric $n\times n$ sub-stochastic matrix, that is $P_{i,j}\ge 0$ for all $i,j\in[n]$, $\sum_jP_{i,j}\le 1$ for all $i\in[n]$, and such that ${\rm Id} - P$ is invertible. We also assume that $P$ is irreducible, that is for all $i,j$ there exists $k\in\N$ such that $P^k_{i,j}>0$.  A well studied example in equilibrium statistical mechanics is obtained by taking $P$ as the transition matrix of the random walk on $\Z^d$ that is killed upon exiting a finite domain $V\subset \Z^d$, 
see e.g.\ \cite[Chapter 8]{friedli2017statistical}. 

Thus, we fix a positive definite covariance matrix $\Gamma=(\Gamma_{ij})_{1\le i,j\le n}$ as above, and consider the probability space
\begin{equation}\label{eq:dgff}
\Omega := \R^n,\qquad \cF :=  \cB(\R^n), \qquad \PP({\rm d}x):=\frac{1}{\sqrt{(2\pi)^{n}\det(\Gamma)}}\exp\left(-\frac{x^\top\Gamma^{-1}x}{2}\right){\rm d}x.
\end{equation}
As motivated in Section \ref{sec:motivation}, we assume to be given a probability vector $(\theta_A)_{A\subset[n]}$, and we seek to estimate  the optimal constants  $\lambda$ and $\kappa$ in the inequalities
\begin{eqnarray}
\label{eq:GBVT}
\forall f\in L^2,\qquad 
\lambda\,\var(f) & \le & \sum_{A\subset[n]}\theta_A\EE\left[\var_A(f)\right],\\
\label{eq:GBET}
\forall f\in L\log L,\qquad 
\kappa\,\ent(f) & \le & \sum_{A\subset[n]}\theta_A\EE\left[\ent_A(f)\right].
\end{eqnarray}
The constant $\lambda$  coincides with the \emph{spectral gap} of the \emph{weighted block dynamics}, that is the Markov chain where at each step a block $A\subset[n]$ is picked with probability $\theta_A$ and the variables $(X_i, i\in A)$ are resampled according to the conditional distribution $\PP(\cdot|X_{A^c})$. Indeed, \eqref{eq:GBVT} can be written in the form $\lambda\,\var(f) \le \cE(f,f)$ where $\cE$ denotes the Dirichlet form of the weighted block dynamics 
\begin{eqnarray}
\label{eq:dirG}
\forall f,g\in L^2,\qquad 
\cE(f,g) & := & \sum_{A\subset[n]}\theta_A\EE\left[\cov_A(f,g)\right],
\end{eqnarray}
and $\cov_A(f,g)=\EE_A[fg]-\EE_A[f]\EE_A[g]$ is the covariance w.r.t.\ $\PP(\cdot|X_{A^c})$.
On the other hand, the constant $\kappa$ in \eqref{eq:GBET} provides a lower bound on the \emph{modified log-Sobolev} constant, describing the rate of exponential decay of the relative entropy  for the same dynamics in continuous time. Indeed,  the modified log-Sobolev constant $\rho_{\rm MLS}$ is defined as the largest $\rho\ge0$ such that 
\begin{eqnarray}
\label{eq:GMLS}
\forall f\in L\log L,\qquad 
\rho\,\ent(f) & \le & \cE(f,\log f),
\end{eqnarray}
and, since by convexity $\cov_A(f,\log f)\ge \ent_A(f)$ for all $f\in L\log L$, $A\subset[n]$, it follows that $\kappa\le \rho_{\rm MLS}$. Moreover, it is well known that $\lambda\ge\kappa$ and $2\lambda \ge \rho_{\rm MLS}$. 
 We refer e.g.\ to \cite{entropy} for more background and a proof of these standard relations. 
 
 An application of our general framework will allow us to control 
 the 
 fundamental constants $\lambda,\kappa$ in terms of the least eigenvalue of the ``laplacian" $\Delta:={\rm Id} - P$, and to show that they actually coincide for the Glauber dynamics, that is the process obtained when $\theta_A = \frac1n\,{\bf 1}_{|A|=1}$. 
\begin{theorem}
\label{th:gauss}
For any irreducible, symmetric, sub-stochastic matrix $P$ such that ${\rm Id} - P$ is invertible, the Gaussian free field $\PP$ with covariance $\Gamma=({\rm Id} - P)^{-1}$ satisfies 
\begin{eqnarray}\label{eq:mainGauss}
\kappa \ge \delta\,\theta_\star\,,
\end{eqnarray}
where $\delta$ denotes the smallest eigenvalue of the matrix $\Delta={\rm Id} - P$, and $\theta_\star=\min_{1\le i\le n}\sum_{A\ni i}\theta_A$.
Moreover, the Glauber dynamics satisfies
\begin{eqnarray}\label{eq:Glaub}
 \lambda= \kappa =\frac\delta{n}\,.
\end{eqnarray}
\end{theorem}
Theorem \ref{th:gauss} can be seen as the Gaussian Free Field version of Shearer inequality, which is shown to be optimal for the Glauber dynamics. In fact, as the proof will show, in that case the inequality (\ref{eq:GBVT}) is saturated by a linear function $f(x)=z^\top x$, where $z\in\R^n$ is such that $\Delta z= \delta\,z$. Thus the spectral gap is achieved at a linear function. From the above mentioned inequalities $\kappa\le \rho_{\rm MLS}\le 2\lambda$, this also shows that the Glauber dynamics has a modified log-Sobolev constant satisfying 
 \begin{eqnarray}
\label{def:mlsGauss}
\frac\delta{n}\;\le\;\rho_{\rm MLS}\;\le\; \frac{2\delta}{n}\,\,.
\end{eqnarray}
Quantifying the speed of convergence to equilibrium of the
Glauber dynamics for Gaussian free fields is  a natural problem which, to the best of our knowledge, has only been investigated in the two-dimensional square lattice \cite{ganguly2023cutoff}, that is when $P$ is the transition matrix of the random walk on $\Z^2$ that is killed upon exiting a finite domain $V\subset \Z^2$. As a consequence of Theorem \ref{th:gauss}, one can compute explicitly both the spectral gap and the entropy contraction constant $\kappa$ of the lattice Gaussian free field, for 
any dimension $d\ge 1$, and for any domain $V$ of the form $V=[n_1]\times\cdots\times[n_d]\subset \Z^d$, 
since in that case $\lambda=\kappa = \delta/n$, and 
\begin{eqnarray*}
\delta \; = \; \frac{2}{d}\sum_{k=1}^d\left[1-\cos\left(\frac{\pi}{n_k+1}\right)\right],
\end{eqnarray*}
which behaves like $2\pi^2/n^2$ in the regime $n_1=n_2=\cdots=n_d=n$ and $n\to\infty$. 
 
 \begin{remark}\label{rem:gaussian}
 Combining Theorem \ref{th:gauss} with Theorem \ref{th:dual} one obtains equivalent B-L inequalities for the gaussian measures. The latter were in fact already known, since they can be directly related to the classical B-L inequalities for Lebesgue measures, see e.g.\ \cite[Fact 6]{barthe2011correlation}. A direct proof of the gaussian B-L inequality is also possible along the lines of \cite{lehec2014short}, see \cite[Theorem 3]{courtade2024rigid}. Thus, one could have obtained Theorem \ref{th:gauss} from these known estimates. However, the purpose of  Theorem \ref{th:gauss} is to illustrate that an entirely different approach based on curvature can be used to obtain optimal entropy factorization estimates, and, in principle, it can yield an alternative proof of the classical B-L inequalities as well. One caveat is that we assume the specific form $\Gamma=({\rm Id}-P)^{-1}$ for the covariance matrix of our gaussian measure, whereas it would be desirable to have the result for all positive definite matrices $\Gamma$. For technical reasons, in this more general setup 
 at the moment we can only prove the desired estimate 
 up to a factor $2$, that is we can show that $\sigma\le 2\kappa\le 2\sigma$, for an explicit spectral quantity $\sigma=\sigma(\Gamma,\theta)$, see Remark \ref{rem:sigma} for more details. We believe it should be possible to obtain $\kappa=\lambda=\sigma$ for all $(\Gamma,\theta)$, as we prove in the case of $\Gamma=({\rm Id}-P)^{-1}$ with Glauber dynamics in Theorem \ref{th:gauss}, 
 thus obtaining full applicability of our argument to derive 
 Gaussian and classical B-L inequalities.   
 \end{remark}

 \paragraph{Permutations and weighted block shuffles.}
 
The next example we consider is the case where $\PP$ is the uniform measure over $S_n$, the set of permutations of $[n]$:
 \begin{equation}\label{eq:perm0}
\Omega := S_n,\qquad 
\PP(\sigma):=\frac{1}{n!}\,,\quad \sigma\in\Omega\,.
\end{equation}
As above, we consider the weighted block dynamics associated to a given probability vector $(\theta_A, A\subset[n])$, with Dirichlet form formally given exactly by the expression in \eqref{eq:dirG}.
If we think of the configuration $\sigma\in\Omega$ as representing the positions of $n$ cards, the weighted block dynamics is the Markov chain where at each step a block $A$ is chosen with probability $\theta_A$
and all cards located in $A$ get uniformly reshuffled among themselves.  We refer to this as the $\theta$-weighted block shuffle.  The case $\theta_A=\frac1{\binom{n}2}{\bf 1}_{|A|=2}$ is the well known random transposition dynamics.

\begin{theorem}\label{th:permutations}
For any choice of the probability vector $(\theta_A, A\subset[n])$,
\begin{eqnarray}
\label{eq:perm}
\forall f\in L\log L,\qquad 
\theta_{\star\star}\,\ent(f) & \le & \sum_{A\subset[n]}\theta_A\EE\left[\ent_A(f)\right],
\end{eqnarray}
where $\theta_{\star\star}:=\min_{1\le i<j\le n}\sum_{ A\supset\{i,j\}}\theta_A$. 
\end{theorem}
The estimate \eqref{eq:perm} is equivalent to the subadditivity statement 
\begin{eqnarray}
\label{standardizedperm}
\forall f\in L\log L,\qquad\sum_{A\subset[n]}\theta_A\,\ent\left(\EE_A[f]\right) & \le & (1-\kappa)\ent(f)\,,
\end{eqnarray}
with $\kappa = \theta_{\star\star}$, and can be used to obtain equivalent B-L inequalities via Theorem \ref{th:dual}. 
As we observed in \eqref{eq:GMLS}, the bound \eqref{eq:perm} implies $\rho_{\rm MLS} \ge \theta_{\star\star}$, providing a lower bound on the rate of exponential decay to stationarity of the continuous time version of the $\theta$-weighted block shuffle. 
  
It is interesting to note that the estimate \eqref{eq:perm} is equivalent to the B-L bounds proved in \cite[Proposition 21]{barthe2011correlation}, which were obtained by an entirely different approach, namely using monotonicity along semigroup interpolations for the corresponding B-L inequality. 
On the other hand, in the ``mean field'' case where 
$\theta_A$ depends only on the size of $A$, that is $\theta_A=\varphi(|A|)$ for some function $\varphi\ge 0$, optimal bounds on the constant $\kappa$ in  \eqref{standardizedperm} were obtained recently in \cite{BristielC}. In particular, \cite[Theorem 1.9]{BristielC} shows that for any $\ell\in\{1,\dots,n\}$,  if $\theta_A =  \frac1{\binom{n}{\ell}}\,{\bf 1}_{|A|=\ell}$, then 
\eqref{standardizedperm}  holds with the optimal constant
\[
\kappa=\kappa_{\ell,n} := \frac{\log(\ell!)}{\log(n!)},
\]
and that the only extremal functions are Dirac masses. 
Here one has $\theta_{\star\star}= \frac{\ell(\ell-1)}{n(n-1)} $ which is smaller than $\kappa_{\ell,n}$ for $\ell\in\{2,\dots,n-1\}$, and therefore Theorem \ref{th:permutations} is not optimal for such homogeneous weights. However, 
Theorem \ref{th:permutations}, as well as  \cite[Proposition 21]{barthe2011correlation}, addresses the entropy factorization problem for arbitrary weights $(\theta_A)$, and in this generality they may provide the best known bounds.

 \paragraph{Entropy of $\ell^n_p$--spherical marginals.}
 A celebrated result of Carlen-Lieb-Loss \cite{carlen2004sharp} states that the uniform probability measure on the unit sphere $\bbS^{n-1}$ satisfies a subadditivity estimate with a factor 2, independently of $n$, with respect to the bound satisfied by a 
 product measure, 
 thus quantifying the departure from independence for this distribution. An equivalent formulation of this is that when $\PP$ is the uniform probability measure on $\bbS^{n-1}$, then the inequality \eqref{standardized}, taking the homogeneous weights $\theta_A : = \frac1n{\bf 1}_{|A|=n-1}$,  holds with constant 
 \begin{align}\label{eq:CCL}
 \kappa = 1-\frac2{n} \,, 
 \end{align}
 for all $n\ge 2$. Moreover, the same authors also show that this constant is optimal. Extensions of this result were later proposed by \cite{barthe2006entropy} and \cite{barthe2011correlation}. In particular, 
 it was shown in \cite[Proposition 11]{barthe2011correlation} that for any choice of the probability vector $\theta$, one has the estimate \eqref{standardized} with $\kappa = \theta_{\star\star}$, where $\theta_{\star\star}:=\min_{1\le i<j\le n}\sum_{ A\supset\{i,j\}}\theta_A$. Note that this holds for arbitrary $\theta$ and it recovers the optimal bound \eqref{eq:CCL} in the ``all but one" case $\theta_A : = \frac1n{\bf 1}_{|A|=n-1}$. The proof of these bounds was based on an extension of the approach introduced in \cite{carlen2004sharp}, which used monotonicity along semigroup interpolations for the corresponding B-L inequality. In the special case $\theta_A : = \frac1{\binom{n}2}\,{\bf 1}_{|A|=2}$, the associated weighted block dynamics \eqref{eq:dirG} is often referred to as the Kac walk on the sphere.     
In this case the bound $\kappa = \theta_{\star\star}$ gives  
\begin{align}\label{eq:stst2}
 \kappa = \frac2{n(n-1)}\,.
 \end{align}
Letting $\rho_{\rm MLS}$ denote the modified log-Sobolev constant from  \eqref{eq:GMLS}, the simple bound  $\rho_{\rm MLS} \ge \kappa$ allows one to obtain $\rho_{\rm MLS} \ge \frac2{n(n-1)}$, which recovers (and actually improves by a factor 2) an estimate previously shown by Villani in \cite[Theorem 6.1]{villani2003cercignani} by a different type of semigroup interpolation. 

Here we discuss how our alternative method could be applied to obtain the same estimates, in the general setting of $\ell^n_p$--spheres, for all $p>0$. Namely, for $p>0$, define the $\ell^n_p$--sphere $\bbS^{n-1}_p$ as
 \begin{align}\label{eq:lp}
\Omega\;=\; \bbS^{n-1}_p\;:=\;\{x\in\R^n:\;\|x\|_p = 1\}\,,\qquad \|x\|_p = \left(\sum_{i=1}^n|x_i|^p\right)^{1/p}\,.
 \end{align}
We let $\PP$ denote the law of the vector $X=(X_1,\dots,X_n)$ obtained as
 \begin{align}\label{eq:gabe}
X = \frac{(G_1,\dots,G_n)}{S_p(G)^{1/p}}\,,\qquad S_p(x):=\sum_{i=1}^n|x_i|^p\,,
 \end{align}
where $G=(G_1,\dots,G_n)$ is a vector of i.i.d.\ real random variables with density proportional to $e^{-|x|^p}$, $x\in\R$. 
$\PP$ is also called the \emph{cone measure} on $\bbS^{n-1}_p$; see e.g.\ \cite{barthe2005probabilistic}. A characteristic property of this construction is that $X$ and $S_p(G)$ are independent. Note that, for $p=1,2$, $\PP$  coincides with the uniform distribution on $\bbS^{n-1}_1$, and $\bbS^{n-1}_2=\bbS^{n-1}$ respectively. 
We believe that it is possible to prove the following statement as a consequence of Theorem \ref{th:main}.
\begin{conjecture}\label{claim}
For all $p>0$, all probability vectors $\theta$, the cone measure $\PP$ on $\bbS^{n-1}_p$ satisfies   
\begin{eqnarray}
\forall f\in L\log L,\qquad 
\theta_{\star\star}\,\ent(f) & \le & \sum_{A\subset[n]}\theta_A\EE\left[\ent_A(f)\right].
\end{eqnarray}
\end{conjecture}
For $p=2$ this coincides with the result in \cite[Proposition 11]{barthe2011correlation}. A related estimate for all $p>0$, under additional symmetry assumptions on $f$, was obtained in \cite[Corollary 16]{barthe2011correlation}.
Note that the above statement is formally equivalent to Theorem \ref{th:permutations}. In fact, to prove the conjecture one may proceed exactly as in our proof of  Theorem \ref{th:permutations}, namely by exhibiting a metric on $\Omega$ for which Assumptions (i) and (ii) in Theorem \ref{th:main} are satisfied with $\kappa = \theta_{\star\star}$, when the Markov operators are given by the conditional expectations  $T_A f = \EE_Af$. As we show in 
Section \ref{sec:lp}, this can be achieved by taking the special metric 
\[
\dist(x,y):=\sum_{i=1}^n\left||x_i|^p -  |y_i|^p\right| + \sum_{i=1}^n{\bf 1}_{\{x_iy_i<0\}}\,.
\]
Assumption (iii) however fails in this case. As in the case of Theorem \ref{th:gauss}, the proof of Conjecture \ref{claim} would be complete if we could establish a regularizing procedure as described  in Remark \ref{rk:regul}, but we were unable to find the needed regularizing kernels for this model, since the natural diffusion process on $\bbS^{n-1}_p$ does not seem to have nonnegatve sectional curvature with respect to our metric.  


 \paragraph{Down-up walk for uniform $n$-sets.}
 Let $\bS$ be a finite space with cardinality $N=|\bS|$. Given $n\in\{1,\dots,N-1\}$, let $\Omega=\binom{\bS}{n}$ denote set of all subsets of $\bS$ with cardinality $n$, and call $\PP$  the uniform distribution over $\Omega$.
  For a fixed  integer $k\in\{1,\dots,n\}$, we let $T_k$ denote the Markov operator associated to the following resampling step. Given $X\in\Omega$, let $X_{-k}$ denote a uniformly random subset of $X$ with cardinality $n-k$, and, given $X_{-k}$, let $X'\in\Omega$ denote a uniformly random supset of $X_{-k}$ with cardinality $n$. Thus $X'$ is obtained from $X$ by first removing $k$ elements chosen uniformly at random from $X$, thus obtaining $X_{-k}$, and then by adding $k$ elements chosen uniformly at random from $\bS\setminus X_{-k}$. The corresponding Markov operator $T_k$ is then written as 
\[
T_kf(X) = \sum_{X'\in\Omega}T_k(X,X')f(X')\,,\qquad X\in\Omega,
\] 
where $T_k(X,X')$ denotes the probability of the transition from $X$ to $X'$ as a result of the the above described step. The Markov chain associated to $T_k$ is also called the $(n\leftrightarrow n-k)$ \emph{down-up walk} on uniform $n$-sets. By symmetry, $T_k=T_k^\star$ and the walk is $\PP$-reversible. 

The $(n\leftrightarrow n-k)$ down-up walk on uniform $n$-sets is a special case of the $(n\leftrightarrow n-k)$ down-up walk on the \emph{bases of a matroid}, an extensively studied Markov chain which is known to satisfy an entropy contraction with rate $\frac{k}{n}$, for all $k$, that is, for any $f:\Omega\mapsto\R_+$, 
\[
\ent(T_k f)\le \left(1-\frac{k}{n}\right)\,\ent f\,,
\]
see \cite{cryan2021modified,anari2022entropic}. Somewhat surprisingly, as a simple consequence of our main result (Theorem \ref{th:main}), we obtain an improvement for the $(n\leftrightarrow n-k)$ down-up walk on uniform $n$-sets, and show that for any $k$ one has entropy contraction with the strictly larger rate $\frac{k}{n} + \frac{k}{N-(n-k)}\left(1- \frac{k}{n}\right)$. We formulate this in  the following slightly more general terms. 
\begin{theorem}\label{th:downup}
For any $\theta=(\theta_1,\dots,\theta_n)$, with $\theta_k\ge0$ and $\sum_{k=1}^n\theta_k=1$, for any $f:\Omega\mapsto\R_+$, 
\begin{eqnarray}
\sum_{k=1}^n\theta_k\ent(T_k f) \leq (1-\kappa)\ent f,
\label{eq:downup}
\end{eqnarray}
with $\kappa := \kappa_0(\theta) + \kappa_1(\theta)$, where 
\begin{eqnarray}
\kappa_0(\theta):=\sum_{k=1}^n\theta_k\,\frac{k}{n}\,,\quad 
\kappa_1(\theta):=\sum_{k=1}^n\theta_k\,\frac{k}{N-(n-k)}\left(1- \frac{k}{n}\right)\,.
\label{eq:downupk}
\end{eqnarray}
In particular, for any $k=1,\dots,n$, the $(n\leftrightarrow n-k)$ down-up walk for uniform $n$-sets has entropy contraction with rate $\frac{k}{n} + \frac{k}{N-(n-k)}\left(1- \frac{k}{n}\right)$. 
\end{theorem}
The proof is based on showing that Assumptions (i) and (ii) of Theorem \ref{th:main} are satisfied with the required constants. In Section \ref{sec:downup} we provide the simple coupling argument for these curvature bounds. 

 \paragraph{Langevin diffusion in a convex potential.} 
 Finally, we consider the  $n-$dimensional Langevin diffusion in a convex potential, that is the stochastic differential equation
\begin{eqnarray}
\label{def:langevin}
\dd X_t & = & -\nabla V(X_t)\dd t+\sqrt{2}\dd B_t,
\end{eqnarray}
where $B=(B_t)_{t\ge 0}$ is a standard $n-$dimensional Brownian motion and $V\colon\R^n\to\R$   a smooth (say, twice continuously differentiable) function, which is assumed to be $\rho-$convex:
\begin{eqnarray}
\label{assume:convex}
\mathrm{Hess}(V) & \ge & \rho\,\mathrm{Id},
\end{eqnarray}
for some constant $\rho>0$.
The classical theory ensures that there is a unique strong solution $X=(X_t)_{t\ge 0}$ starting from any fixed condition $x\in\R^n$, and that the formula $(P_tf)(x)  :=  \EE_x\left[f(X_t)\right]$ defines a self-adjoint Markov semi-group $(P_t)_{t\ge 0}$ on $L^2(\Omega,\cF,\PP)$, where
\begin{eqnarray}
\Omega=\R^n,\qquad \cF=\cB(\R^n),\qquad \PP({\rm d}x) & \propto & e^{-V(x)}\dd x.
\end{eqnarray}
The following fundamental result on the relative entropy decay for Langevin diffusions is one of the most emblematic applications of the Bakry-\'Emery theory. It is equivalent to a log-Sobolev inequality for the measure $\PP$, with optimal constant in the case where $V$ is quadratic, see e.g.\ \cite{bakry2014analysis}. 
\begin{theorem}
\label{th:BE}
Under assumption (\ref{assume:convex}), one has
\begin{eqnarray*}
\ent(P_t f) & \le & e^{-2\rho t}\,\ent(f),
\end{eqnarray*}
for all times $t\ge 0$ and all functions $f\in L\log L$. 
\end{theorem}
In Section \ref{sec:langevin} we show how to use our framework to obtain an elementary probabilistic proof of Theorem \ref{th:BE}.

\bigskip

 \begin{ackno} {\em 
We thank Thomas Courtade for helpful conversations on the B-L inequalities, as well as Francesco Pedrotti and Sam Power for many relevant comments and references.  
P.C. warmly thanks the research center CEREMADE at Paris Dauphine for the kind hospitality. }
 \end{ackno}
 
\section{Proof of Theorem \ref{th:main}}
\subsection{Step 1: Lipschitz contraction}

Our first step consists in converting the curvature assumptions into a crucial Lipschitz contraction property for the 
non-linear operator $\Lambda\colon L^\infty\to L^\infty$ defined by
\begin{eqnarray*}
\Lambda f & := & \sum_{i=1}^M\theta_i T_i\log T_i^\star\exp f.
\end{eqnarray*}

\begin{proposition}[Lipschitz contraction]\label{pr:lip}Under Assumptions $(i)-(ii)$, we have
\begin{eqnarray}
\label{assume:lip}
\forall f\in L^\infty,\qquad \Lip\left(\Lambda f\right) & \le & (1-\kappa)\,\Lip(f).
\end{eqnarray}
 \end{proposition}
 We start with an elementary remark that will be used several times throughout the proof. 
\begin{claim}[Null sets]\label{claim:null}If $T$ is a measure-preserving transition kernel on $(\Omega,\cF,\PP)$, and if $N\subset\Omega$ is a null set, then so is $\left\{x\in\Omega\colon T(x,N)>0\right\}$. 
\end{claim}
\begin{proof}
Simply choose $f={\bf 1}_N$ in (\ref{def:statio}) to obtain
$
\int_\Omega T(x,N)\PP(\dd x)  =  \PP(N)  =  0.
$
\end{proof}
\begin{proof} [Proof of Proposition \ref{pr:lip}]
Fix $f\in L^\infty$. By definition, there is a null set $N\subset\Omega$ such that
\begin{eqnarray*}
\forall (x,y)\in(\Omega\setminus N)^2,\qquad |f(x)-f(y)| & \le & \Lip(f)\dist(x,y).
\end{eqnarray*}
Claim \ref{claim:null} ensures that 
$
N' :=  \bigcup_{i=1}^M\left\{x\in\Omega\colon T_i^\star(x,N)>0\right\}
$
is also a null set. Upon enlarging it if necessary, we may assume that $N'$ also contains the null sets appearing in Assumption (i). Now, fix two points $x,y\in\Omega\setminus N'$, and an index $1\le i\le n$. By assumption, there is a coupling $(X^\star,Y^\star)$ of $T_i^\star(x,\cdot)$ and $T_i^\star(y,\cdot)$ such that almost-surely,
$
\dist(X^\star,Y^\star)\le \ell_i\dist(x,y)$ and $X^\star,Y^\star\notin N$. In particular, this implies  that almost-surely, 
\begin{eqnarray*}
f(X^\star) & \le & f(Y^\star)+\ell_i\Lip(f)\dist(x,y),
\end{eqnarray*} 
Taking exponentials, then expectations, and then logarithms, we arrive at the key inequality
\begin{eqnarray}
\label{step1}
(\log T_i^\star \exp f)(x) & \le & (\log T_i^\star \exp f)(y)+\ell_i\Lip(f)\dist(x,y),
\end{eqnarray}
valid for any $1\le i\le M$ and any  $x,y\in \Omega\setminus N' $. Now, invoking Claim \ref{claim:null} again, we know that $N'' :=  \bigcup_{i=1}^M\left\{x\in\Omega\colon T_i(x,N')>0\right\}
$ is a null set and, upon enlarging it if needed, we may assume that it contains the null set appearing in Assumption (ii). Consider a pair $(x,y)\in (\Omega\setminus N'')^2$, and let $(X,Y)$ be any coupling  of $T_i(x,\cdot)$ and $T_i(y,\cdot)$. Then, $X,Y$ are both in $\Omega\setminus N'$ almost-surely, so that (\ref{step1}) holds almost-surely with $x,y$ replaced by $(X,Y)$. Taking expectations, we obtain
\begin{eqnarray*}
(T_i\log T_i^\star \exp f)(x) & \le & (T_i\log T_i^\star \exp f)(y) +\ell_i\Lip(f)\EE\left[\dist(X,Y)\right].
\end{eqnarray*}
Since this is true for any choice of the coupling $(X,Y)$,  we may  replace $\EE\left[\dist(X,Y)\right]$ with $W_1\left(T_i(x,\cdot),T_i(y,\cdot)\right)$. Finally, we simply multiply through by $\theta_i$ and sum over $1\le i\le M$. In view of Assumption (ii) and our definition of $\Lambda$, we obtain
\begin{eqnarray*}
(\Lambda f)(x) & \le & ( \Lambda f)(y) +(1-\kappa)\,\Lip(f)\dist(x,y).
\end{eqnarray*}
Since this is true for any $x,y\in\Omega\setminus N''$, the  claim is proved. 
\end{proof}

\subsection{Step 2: variance contraction}

With Proposition \ref{pr:lip} on hand, the proof of Theorem \ref{th:main} boils down to establishing that the Lipschitz contraction $\Lip\left(\Lambda f\right)  \le  (1-\kappa)\Lip(f)$ and the regularity assumption (iii) together imply the desired entropy contraction (\ref{eq:main}). We first establish a weaker statement, obtained by replacing entropies with variances.

\begin{proposition}[Variance contraction]\label{pr:variance}Under Assumption (iii), the contraction (\ref{assume:lip}) implies
\begin{eqnarray}
\forall f\in L^2,\qquad \sum_{i=1}^M\theta_i\var\left(T_i^\star f\right) & \le & (1-\kappa)\var\left(f\right).
\end{eqnarray}
\end{proposition}
\begin{proof}Without loss of generality, we assume that $f\in L^\infty$ and that $\EE[f]=0$ and $\EE[f^2]=1$. Using the uniform approximation $e^{\varepsilon f}=1+\varepsilon f+o(\varepsilon)$ as $\varepsilon\to 0$, we have 
\begin{eqnarray*}
\frac{1}{\varepsilon}\Lambda(\varepsilon f) & \xrightarrow[\varepsilon\to 0]{L^2} &  Tf,
\end{eqnarray*}
where we have introduced the operator $T:= \sum_{i=1}^M\theta_iT_iT_i^\star$. It follows that $T$ inherits the Lipschitz contraction property (\ref{assume:lip}) from $\Lambda$, i.e. 
\begin{eqnarray}
\label{assume:weaklip}
\Lip\left(T f\right) & \le & (1-\kappa)\,\Lip(f).
\end{eqnarray}
Now, $T$ is clearly a non-negative self-adjoint Markov operator on the Hilbert space $L^2$, so the spectral theorem provides us with a Borel probability measure $\mu$ on $[0,1]$ such that
\begin{eqnarray}
\forall n\in\N,\qquad  \int_{0}^{1} \lambda^n\mu({\rm d}\lambda) & =&  \langle f,T^n f\rangle.
\end{eqnarray}
Iterating the Lipschitz contraction (\ref{assume:weaklip}), we find
\begin{eqnarray*}
 \int_{0}^{1} \lambda^{2n}\mu({\rm d}\lambda)  & = & \var(T^n f) \\ & = & \frac{1}{2}\int_{\Omega^2}\left[(T^nf)(x)-(T^nf)(y)\right]^2\PP(dx)\PP(dy)\\
 & \le & \frac{\Lip(T^nf)^2}{2}\int_{\Omega^2}\dist^2(x,y)\PP(dx)\PP(dy)\\
 & \le & (1-\kappa)^{2n-2}\,\frac{\Lip(Tf)^2}{2}\,\int_{\Omega^2}\dist^2(x,y)\PP(dx)\PP(dy).
\end{eqnarray*}
Note that the right-hand side is finite thanks to Assumption (iii) and (\ref{assume:squareint}). We may thus safely raise both sides to the power $1/(2n)$ and send $n\to\infty$ to conclude that  the measure $\mu$ is actually supported on $[0,1-\kappa]$. In particular, we have
\begin{eqnarray*}
\langle f,T f\rangle  & = & \int_{0}^{1} \lambda\mu({\rm d}\lambda) \ \le \ 1-\kappa.
\end{eqnarray*}
By definition of $T$, the left-hand side is exactly $\sum_{i=1}^M\theta_i\var(T_i^\star f)$, and the claim is proved.
\end{proof}
\subsection{Step 3: entropy contraction}
We henceforth fix a parameter $m\in(0,\infty)$ and consider the compact set $K\subset L^2$ formed by those  functions $f$ satisfying $\|f\|_\infty\le m$ and $\EE[e^{f}]=1$. We then introduce the  constant
\begin{eqnarray}
\rho & := & \sup_{f\in K\setminus\{0\}}\cH(f),\qquad\textrm{ where }\qquad \cH(f):=\frac{\sum_{i=1}^M\theta_i\ent(T_i^\star e^f)}{\ent(e^f)}.
\end{eqnarray}
We will prove that $\rho\le 1-\kappa$, independently of $m$. This  establishes the desired entropy contraction property (\ref{eq:main}) for all positive measurable functions that are bounded away from $0$ and $\infty$. By a straightforward approximation argument, the conclusion then extends  to all functions in $L\log L$, and our proof of Theorem \ref{th:main} is complete.
\begin{proposition}[Entropy contraction]\label{pr:rho}Under Assumption (iii), the Lipschitz contraction (\ref{assume:lip}) implies that $\rho\le 1-\kappa$, independently of the parameter $m$. 
\end{proposition}
An essential ingredient is the following structural property of  optimizers of $\cH$.
\begin{lemma}[Optimizers]\label{lm:optimizers}If $f\in K\setminus\{0\}$ satisfies $\cH(f)=\rho$, then  there is  $\alpha\in\R$ such that 
\begin{itemize}
\item $\Lambda f \le \rho f  + \alpha$ a.-s. on the set $\{f\ne m\}$.
\item $\Lambda f \ge \rho f + \alpha$ a.-s. on the set $\{f\ne -m\}$.
\end{itemize}
In particular, we must have $\rho\, \Lip(f)\le \Lip(\Lambda f)$.
\end{lemma}
\begin{proof}Fix $f\in K\setminus\{0\}$ such that $\cH(f)=\rho$, and let us introduce the short-hands
\begin{eqnarray*}
A:=\{f\ne m\},\qquad B:=\{f\ne -m\}\qquad\textrm{ and }\qquad 
\Psi & := & \Lambda f -\rho f.
\end{eqnarray*}
With this notation at hands, the two items in the statements are rewritten as 
\begin{eqnarray}
\label{claim:op}
\mathrm{ess\,\sup}_A \Psi & \le &  \mathrm{ess\,\inf}_B \Psi.
\end{eqnarray}
Assume for a contradiction that (\ref{claim:op}) fails. This means that we can find $\beta\in\R$ such that 
\begin{eqnarray*}
\PP\left(A\cap \{\Psi>\beta\}\right)>0 & \textrm{ and } & \PP\left(B\cap \{\Psi<\beta\}\right)>0.
\end{eqnarray*}
By monotone convergence, this remains true with   $A$ and $B$ replaced by $A':=\{f\le m-\delta\}$ and  $B':=\{f\ge-m+\delta\}$, provided  $\delta>0$ is small enough. Now,  consider the function 
\begin{eqnarray}
\label{def:h}
h & := & \frac{{\bf 1}_{A'\cap \{\Psi>\beta\}}}{\PP(A'\cap\{\Psi>\beta\})}-\frac{{\bf 1}_{B'\cap \{\Psi<\beta\}}}{\PP(B'\cap\{\Psi<\beta\})}.
\end{eqnarray}
Note that $h\in L^\infty$, with $\EE[h]=0$ and $h\le 0$ on $(A')^c$ and $h\ge 0$ on $(B')^c$.
Those properties guarantee that the function 
$
f_\varepsilon  :=  \log\left(e^{f}+\varepsilon h\right)
$
is in $K$ for all small enough $\varepsilon>0$, and a  Taylor expansion  gives
\begin{eqnarray*}
\cH(f_\varepsilon) & = & \cH(f)+\frac{\varepsilon \langle h, \Psi\rangle}{\ent(e^f)}+o(\varepsilon) \qquad\textrm{ as }\qquad \varepsilon\to 0.
\end{eqnarray*}
Thus, the maximality of $\cH$ at $f$ imposes $\langle h, \Psi\rangle\le 0$. In view of (\ref{def:h}), this reads
\begin{eqnarray*}
\EE\left[\Psi|A'\cap\{\Psi>\beta\}\right] & \le & \EE\left[\Psi|B'\cap\{\Psi<\beta\}\right].
\end{eqnarray*}
This inequality is clearly self-contradictory, and (\ref{claim:op}) is proved. This means that there is a null set $N\subset \Omega$ such that for all $(x,y)\in (A\setminus N)\times (B\setminus N)$, we have $\Psi(x)\le\Psi(y)$, hence
\begin{eqnarray*}
\rho f(y)-\rho f(x) & \le & |(\Lambda f)(y)-(\Lambda f)(x)|.
\end{eqnarray*}
 On the other hand, this inequality  trivially holds when $(x,y)\in \Omega^2\setminus (A\times B)$,  because the left-hand side is then non-positive. The conclusion $\rho\, \Lip(f)\le \Lip(\Lambda f)$ follows.
\end{proof}
\begin{proof}[Proof of Proposition \ref{pr:rho}]By definition of $\rho$, there is a sequence $(f_k)_{k\ge 1}$ in $K\setminus\{0\}$ such that 
\begin{eqnarray}
\label{rholim}
\rho & = & \lim_{k\to\infty}\cH(f_k).
\end{eqnarray}
Upon extracting a subsequence if needed, we may further assume that $(f_k)_{k\ge 1}$ has an almost-sure limit $f$. Note that $f$ automatically inherits the properties $\|f\|_\infty\le m$ and $\EE[e^f]=1$. If $f$ is not a.s.\ zero, then $f\in K\setminus\{0\}$ and 
$\cH(f)=\rho$, so Lemma \ref{lm:optimizers} and  (\ref{assume:lip}) together yield
 \begin{eqnarray*}
\rho\,\Lip(f) & \le  & \Lip\left(\Lambda f\right) \ \le \ (1-\kappa)\,\Lip(f).
\end{eqnarray*} 
Moreover,  Assumption (iii) ensures that the middle term is finite, so that $\Lip(f)<\infty$. Therefore, we may safely simplify through by $\Lip(f)$ to obtain $\rho\le 1-\kappa$, as desired. Consider now the degenerate case where $f=0$ almost surely. We can then write $e^{f_k}=1+h_k$ with $\|h_k\|_\infty\le e^m$, $\EE[h_k]=0$ and $h_k\to 0$ as $k\to\infty$. But then, a Taylor expansion yields
\begin{eqnarray*}
 \ent\left(e^{f_k}\right) & \sim & \frac{1}{2}\var(h_k)\\
 \sum_{i=1}^M\theta_i\ent\left(T_i^\star e^{f_k}\right) & \sim & \frac{1}{2}\sum_{i=1}^M\theta_i\var(T^\star_i h_k),
 \end{eqnarray*} 
where the notation $a_k\sim b_k$ means that $a_k/b_k\to 1$ as $k\to\infty$. In particular, (\ref{rholim}) becomes
\begin{eqnarray*}
\rho & = & \lim_{k\to\infty}\frac{\sum_{i=1}^M\theta_i\var(T^\star_i h_k)}{\var(h_k)}.
 \end{eqnarray*}
By Proposition \ref{pr:variance}, the right-hand is at most $1-\kappa$, and the proof is complete.
\end{proof}

\section{Applications}\label{sec:appl}
In this section we address the application of Theorem \ref{th:main} to the main examples discussed in Section \ref{sec:applications}.
\subsection{Gaussian Free Fields: Proof of Theorem \ref{th:gauss}}\label{sec:gauss}
Let us first show that the second part of the theorem follows from the first, namely that once we have $\kappa \ge \delta \theta_\star$ for all choices of $\theta$, it follows that for Glauber dynamics one has $\kappa=\lambda=\delta/n$.  Since $\theta_\star=1/n$ in this case, and, in general, $\kappa\le \lambda$, the claim \eqref{eq:Glaub} follows from \eqref{eq:mainGauss} and the following observation.
\begin{lemma}\label{lem:linearGap}
The Glauber dynamics satisfies $\lambda \le \delta/n$. 
\end{lemma}
\begin{proof}
By definition, 
\begin{eqnarray}
\label{PI:bis}
\forall f\in L^2,\qquad 
\lambda\,\var(f) & \le & \cE(f,f) \ = \frac1n\, \EE\left [f\sum_{i=1}^n(f-\EE_if)\right].
\end{eqnarray}
 Let us apply this to the linear observable $f(x)=x^\top z$, where $z\in\R^n$ is an eigenvector of $\Delta$ corresponding to the eigenvalue $\delta$. If $u_j(x)=x_j$, then  $\EE_i[u_j](x)=x_j$ for all $j\neq i$, and $\EE_i[u_j](x)=(Px)_i$ if $j=i$. Therefore,  for all $x\in\R^n$,
\begin{eqnarray*}
\sum_{i=1}^n(f-\EE_i f)(x) & = & \sum_{i=1}^n z_i(\Delta x)_i \ =  \ z^\top \Delta x \ = \ x^\top \Delta z \ = \ \delta f(x).
\end{eqnarray*}
Thus, (\ref{PI:bis}) forces $\lambda\le \delta/n$, as desired.  
\end{proof}
We turn to the proof of \eqref{eq:mainGauss}. Recall the notation introduced in \eqref{eq:dgff}, and 
let $Z\colon \Omega\to\Omega$ denote the identity map on $\Omega$, which forms a random vector with law $\PP$. 
For each set $A\subset[n]$, we consider the  Markov operator $T_A\colon L^2\to L^2$ defined by
\begin{eqnarray}
T_A f & := &  \EE[f|Z_j,j\in \bA],
\end{eqnarray}
where $\bA=[n]\setminus A$. Note that $T_A$ is a self-adjoint, measure-preserving, Markov operator. 
In order to apply Theorem \ref{th:main}, we need to investigate the curvature of the associated transition kernel.
\paragraph{A contractive coupling}
Let us start 
with an explicit  
expression for  the transition kernel associated to our Markov operator $T_A$. 
\begin{lemma}[Explicit representation]\label{lem:lexpl}
For all $f\in L^2$ and  $x\in\R^n$,
\begin{eqnarray}
\label{def:TA}
(T_Af)(x) & = & \int_{\R^n} f\left(M_A z+({\rm Id}-M_A)x\right)\PP({\rm d}z),
\end{eqnarray}
where $M_A$ is the $n\times n$ matrix defined by the blocks
\begin{equation}\label{eq:defM}
\begin{array}{rclrcl}
 M_{A\times A} & = & \mathrm{Id} \qquad\qquad
 & M_{A\times \bA} & = & -\Gamma_{A\times \bA}\,(\Gamma_{\bA\times \bA})^{-1}\\
M_{\bA\times A}& = & 0 & M_{\bA\times\bA} & = & 0.
 \end{array} 
\end{equation}
\end{lemma}
\begin{proof}
From the block definition of $M_A$ above, we can  readily compute $M_A\Gamma$ :
\begin{equation}\label{eq:maga}
\begin{array}{rclrcl}
 (M_A \Gamma)_{A\times A} & = & \Gamma_{A\times A}-\Gamma_{A\times \bA}\,(\Gamma_{\bA\times \bA})^{-1}\,\Gamma_{\bA\times A},\quad\quad
 & (M_A \Gamma)_{A\times \bA} & = & 0,\\
(M_A \Gamma)_{\bA\times A}& = & 0, & (M_A \Gamma)_{\bA\times\bA} & = & 0.
 \end{array} 
\end{equation}
The second line implies that the random vector $M_A Z$ is uncorrelated with $(Z_j\colon j\in \bA)$, hence independent of $\cG:=\sigma(Z_j\colon j\in\bA)$.    
On the other hand, the fact that $M_{A\times A} = \mathrm{Id}$ ensures that the random vector $({\rm Id}-M_A)Z$ is $\cG-$measurable. Thus, the decomposition
\begin{eqnarray*}
\label{dec:orth}
Z & = & ({\rm Id}-M_A)Z+ M_A Z,
\end{eqnarray*}
expresses $Z$ as the sum of a $\cG-$measurable  vector and a $\cG-$independent one. The claim now follows from a standard property of conditional expectation  (see, e.g. \cite[Section 9.10]{MR1155402}).
\end{proof}
To ensure that $T_A$ has non-negative sectional curvature, 
we equip $\R^n$ with the following distorted metric instead of the usual Euclidean norm. Let 
$\psi$ denote an eigenvector of $\Delta$ with eigenvalue $\delta$. Since $P$ is irreducible with nonnegative entries, by Perron-Frobenius theorem we may choose $\psi$ with positive entries: $\psi_i>0$ for all $i\in[n]$. 
We then define the metric
\begin{eqnarray}
\label{def:dista}
\dist(x,y) & := & \sum_{i=1}^n \psi_i | (\Delta x)_i -(\Delta y)_i|  .
\end{eqnarray}
Let $W_1$ and $W_\infty$ denote the Wasserstein distances associated with this metric.
\begin{lemma}\label{lm:gausscurv}
For any $A\subset[n]$, we have
\begin{eqnarray*}
\forall x,y\in\Omega,\qquad W_\infty\left(T_A(x,\cdot),T_A(y,\cdot)\right) & \le & \dist(x,y)-\delta\sum_{i\in A} \psi_i | (\Delta x)_i -(\Delta y)_i|  
\end{eqnarray*}
In particular, 
for any probability vector $(\theta_A,\,A\subset[n])$, one has
\begin{eqnarray*}
\forall x,y\in\Omega,\qquad \sum_{A\subset[n]}\theta_AW_\infty\left(T_A(x,\cdot),T_A(y,\cdot)\right) & \le & (1-\delta\theta_\star)\dist(x,y)\,, 
\end{eqnarray*}
where $\theta_\star=\min_i\sum_{A\ni i}\theta_A$.
\end{lemma}
\begin{proof}
The second inequality is an immediate consequence of the first. To prove the first, fix a subset $A\subset[n]$ and two points $x,y\in\Omega$. Recalling that $Z\sim\PP$, Lemma \ref{lem:lexpl} shows that 
the random  vectors  
\begin{equation*}
\left\{
\begin{array}{rlc}
X & := & ({\rm Id}-M_A)x+M_AZ\\
Y & := & ({\rm Id}-M_A)y+M_AZ.
\end{array}
\right.
\end{equation*}
form a coupling of $T_A(x,\cdot)$ and $T_A(y,\cdot)$. Thus, taking $z=x-y$,  the proof boils down to showing that for all $z\in\R^n$, 
\begin{eqnarray}\label{eq:toprove}
 \sum_{i=1}^n \psi_i | [\Delta({\rm Id}-M_A) z]_i |\;\le\;   \sum_{i=1}^n \psi_i | (\Delta z)_i | - \delta\sum_{i\in A} \psi_i |  (\Delta z)_i |.
\end{eqnarray}
To prove this,  recall that the matrix $M_A\Gamma$ computed in the proof of the previous lemma was symmetric, i.e. $M_A\Gamma= \Gamma M^\top_A$ or equivalently, $\Gamma^{-1}M_A=M^\top_A\Gamma^{-1}$. Combining this with the fact that $M^2_A=M_A$, we obtain $\Gamma^{-1}M_A  =  M^\top_A\Gamma^{-1}M_A$. In particular, $\Delta M_A$ is symmetric and $\Delta({\rm Id}-M_A) =({\rm Id}-M^\top_A)\Delta $. Next, we observe that ${\rm Id}-M^\top_A$ has nonnegative entries. To see this, write 
\[
{\rm Id}-M^\top_A = \begin{pmatrix}
0 & 0\\
(\Gamma_{A^c\times A^c})^{-1}\Gamma_{A^c\times A} & {\rm Id}
\end{pmatrix},
\] 
and note that $\Gamma_{A^c\times A}\Delta_{A\times A} + \Gamma_{A^c\times A^c}\Delta_{A^c\times A}=(\Gamma\Delta)_{A^c\times A} = 0$, or equivalently
\[
(\Gamma_{A^c\times A^c})^{-1}\Gamma_{A^c\times A}\;=\;-\Delta_{A^c\times A}(\Delta_{A\times A})^{-1}.\]
Now, $-\Delta_{A^c\times A} = P_{A^c\times A}$ has nonnegative entries, as well as 
\begin{align}\label{eq:neump}
(\Delta_{A\times A})^{-1}=\sum_{k=0}^\infty (P_{A\times A})^k\,.
\end{align}
This proves that ${\rm Id}-M^\top_A$ has nonnegative entries. As a consequence, 
\begin{align*}
  [\Delta({\rm Id}-M_A) z]_i &\;
 \le\; 
 \sum_{j=1}^n  |({\rm Id}-M^\top_A)_{i,j}(\Delta z)_j |\\ &
\;= \;
\sum_{j=1}^n  ({\rm Id}-M^\top_A)_{i,j}|(\Delta z)_j |
 \; =\;|(\Delta z)_i |- \sum_{j=1}^n |(\Delta z)_j |(M_A)_{j,i}\,.
\end{align*}
Therefore,
\begin{eqnarray}\label{eq:st3}
 \sum_{i=1}^n \psi_i | [\Delta({\rm Id}-M_A) z]_i |\;\le\;   \sum_{i=1}^n \psi_i | (\Delta z)_i | - \sum_{j=1}^n (M_A\psi)_j |  (\Delta z)_j |.
\end{eqnarray}
To conclude, observe that $M_A\psi=
M_A\Gamma \Delta\psi = \delta M_A\Gamma\psi$, and that 
\[
M_A\Gamma = \begin{pmatrix}
(\Delta_{A\times A})^{-1} & 0\\
0 & 0
\end{pmatrix},
\]
where the latter follows from a standard Schur complement computation and \eqref{eq:maga}.  
Thus, $M_A\psi=
\delta (\Delta_{A\times A})^{-1}\psi_A$, and using \eqref{eq:neump}, one finds 
\begin{eqnarray*}
(M_A\psi)_j \;  =\;\delta [(\Delta_{A\times A})^{-1}\psi_A]_j\;\ge\;\delta \,\psi_j\,,\qquad j\in A\,, 
\end{eqnarray*}
and $(M_A\psi)_j = 0$ otherwise. Inserting this in \eqref{eq:st3} we have proved \eqref{eq:toprove}.
\end{proof}
Since $W_1\left(T_A(x,\cdot),T_A(y,\cdot)\right)\le W_\infty\left(T_A(x,\cdot),T_A(y,\cdot)\right)$, Lemma \ref{lm:gausscurv} implies that both Assumptions (i)-(ii) in Theorem \ref{th:main} are verified.
It remains to check the regularity Assumption (iii).  Unfortunately, Assumption (iii) fails, except in the degenerate case where $A=[n]$. To circumvent this, we use the regularization procedure outlined in Remark \ref{rk:regul}, using the Ornstein-Uhlenbeck semi-group. This will conclude the proof of Theorem \ref{th:gauss}.
\paragraph{Ornstein-Uhlenbeck regularization}
Given a parameter $\tau\in(0,1)$, we consider the Markov operator $P_\tau$ defined by
\begin{eqnarray*}
(P_\tau f)(x) & := & \int_{\Omega}f\left(\sqrt{1-\tau}x+\sqrt{\tau}z\right)\PP({\rm d}z).
\end{eqnarray*}
Note that $P_\tau$ is measure-preserving, because the image of $\PP\otimes\PP$ under the map $(x,z)\mapsto \sqrt{1-\tau}x+\sqrt{\tau}z$ is  $\PP$. An easy change of variable yields the alternative expression
\begin{eqnarray}
\label{kernel}
(P_\tau f)(x) & = & \int_{\Omega}p_\tau\left(x,y\right)f\left(y\right)\PP({\rm d}y),
\end{eqnarray}
where the transition kernel $p_\tau\colon\Omega\times \Omega\to\R_+$  is explicitly given by
\begin{eqnarray*}
 p_\tau(x,y) & := & \tau^{-\frac n2}\,\exp\left\{\frac{\sqrt{1-\tau}}{\tau}x^\top \Gamma^{-1}y-\frac {1-\tau}{2\tau}x^\top\Gamma^{-1}x-\frac{1-\tau}{2\tau}y^\top\Gamma^{-1}y\right\}.
\end{eqnarray*}
The fact that this formula is symmetric in $x$ and $y$ readily guarantees that $P_\tau$ is self-adjoint. 
We now verify that $(P_\tau)_{0<\tau<1}$ satisfies the three properties in Remark \ref{rk:regul}.
For the continuity, we need to check that  
for any $f\in L^2$, we have
\begin{eqnarray*}
P_\tau f & \xrightarrow[\tau\to 0]{L^2} &  f.
\end{eqnarray*}
The claim is clear when $f$ is continuous and bounded. In the general case,  consider a sequence  $(f_k)_{k\ge 1}$ of bounded measurable functions that converges to $f$, and  write
\begin{eqnarray*}
\|P_\tau f-f\|_2 & \le & \|P_\tau f_k-f_k\|_2+\|P_\tau f-P_\tau f_k\|_2+\|f_k-f\|_2\\
& \le & \|P_\tau f_k-f_k\|_2+2\|f_k-f\|_2.
 \end{eqnarray*} 
Taking first a $\limsup$ as $\tau\to 0$ and then a limit as $k\to\infty$ concludes the proof. 

The second property  in Remark \ref{rk:regul} is the regularity $\Lip(P_\tau f)<\infty$ for all $f\in L^\infty$, $\tau>0$. Since our metric is comparable to Euclidean norm, this follows from the well known regularizing property of the  Ornstein-Uhlenbeck semi-group, see e.g.\ \cite[Th.\ 4.7.2]{bakry2014analysis}. 

It remains to check the third property  in Remark \ref{rk:regul}, namely non-negative sectional curvature.
\begin{lemma}[Non-negative sectional curvature]\label{lm:OU2}For every $x,y\in \Omega$, we have
\begin{eqnarray*}
W_\infty\left(P_\tau(x,\cdot),P_\tau(y,\cdot)\right)  & \le &  \sqrt{1-\tau}\dist(x,y).
\end{eqnarray*}
\end{lemma}
\begin{proof}With $Z\sim{\rm P}$, we deduce from the definition of $P_\tau$ that the random vectors  
\begin{equation*}
\left\{
\begin{array}{rlc}
X & := & \sqrt{1-\tau}x+\sqrt{\tau}Z\\
Y & := & \sqrt{1-\tau}y+\sqrt{\tau}Z.
\end{array}
\right.
\end{equation*}
form a coupling of $P_\tau(x,\cdot)$ and $P_\tau(y,\cdot)$. Therefore,
\begin{eqnarray*}
\dist(X,Y) & = & \sqrt{1-\tau}\;\sum_{i=1}^n\psi_i\, |[\Delta(x-y)]_i|=\sqrt{1-\tau}\,\dist(x,y) \,,
\end{eqnarray*}
which implies the desired property.
\end{proof}

\begin{remark}\label{rem:sigma}
If instead of the gaussian free field with covariance $\Gamma=({\rm Id} - P)^{-1}$ as above we consider a Gaussian measure \eqref{eq:dgff}
with an arbitrary positive definite $\Gamma$, then a minor modification of our argument yields the following estimate. 
Let $\Sigma$ be the matrix
\begin{eqnarray}
\label{def:S}
\Sigma & := &  \sum_{A\subset[n]}\theta_A  \Gamma^{1/2}M^\top_A\Gamma^{-1}M_A\Gamma^{1/2},
\end{eqnarray}
where $M_A$ is the matrix defined in \eqref{eq:defM}, and call $\sigma = \sigma(\Gamma,\theta)$ the smallest eigenvalue of $\Sigma$. 
Then, taking the metric 
\begin{eqnarray}
\label{def:dist}
\dist(x,y) & := & \sqrt{(x-y)^\top \Gamma^{-1}(x-y)},
\end{eqnarray}
and adapting our previous argument to this case, one finds $W_\infty\left(T_A(x,\cdot),T_A(y,\cdot)\right) \le  \dist(x,y)$, for all $A\subset[n]$, and 
 \begin{eqnarray*}
\sum_{A\subseteq[n]}\theta_A W^2_\infty\left(T_A(x,\cdot),T_A(y,\cdot)\right) & \le & (1-\sigma) \dist^2(x,y).
\end{eqnarray*}
From the crude bound $W_1\le W_\infty$ and Cauchy-Schwarz, one has that Assumptions (i) and (ii) of Theorem \ref{th:main} are satisfied with $\kappa= 1-\sqrt{1-\sigma}$. Moreover, it is not difficult to check, using a linear test function in \eqref{eq:GBVT}, that $\kappa\le \lambda \le \sigma$. In conclusion, the above argument yields the inequality
\begin{eqnarray}
\label{def:eqkala}
1-\sqrt{1-\sigma}\;\le\;\kappa\;\le\;\sigma\,.
\end{eqnarray}
Since $1-\sqrt{1-\sigma}\ge \sigma/2$ this estimate is sharp up to a factor at most $2$. As mentioned in Remark \ref{rem:gaussian}, it should be possible to improve this argument to obtain the sharp result $\kappa=\sigma$ for all positive definite $\Gamma$ and all probability vector $\theta$. The identity $\kappa=\sigma$ could be in fact obtained via Theorem \ref{th:dual} using  known B-L inequalities for Gaussian measures, see \cite[Theorem 3]{courtade2024rigid}.
\end{remark}

  \subsection{Permutations: proof of Theorem \ref{th:permutations}}
Here 
$\PP$ is the uniform measure over $\Omega= S_n$, the set of permutations of $[n]$, see \eqref{eq:perm0}. A permutation $\sigma\in\Omega$ is identified with the configuration $(\sigma_i)_{i\in[n]}$. 
To prove Theorem \ref{th:permutations}, we are going to apply Theorem \ref{th:main} with Markov operators 
  given by the conditional expectations $T_A\colon L^2\to L^2$, $A\subset[n]$, defined by
\begin{eqnarray}
T_A f & := &  \EE_A[f] = \EE[f|\sigma_j,j\in \bA],
\end{eqnarray}
where, as usual,  $\bA=[n]\setminus A$. 
We equip $\Omega=S_n$ with the \emph{transposition distance} $\dist(\sigma,\eta)$, $\sigma,\eta\in\Omega$, which is the minimal number of transpositions needed to turn $\sigma$ into $\eta$. 
Fix a pair of permutations $(\sigma,\eta)$ and let $\sigma'$ and $\eta'$ denote the random permutations with distribution $T_A(\sigma,\cdot)$ and $T_A(\eta,\cdot)$. A coupling of $(\sigma',\eta')$ can be obtained by simply replacing the pair $(\sigma,\eta)$ by the pair $(\sigma',\eta')=(\sigma\pi,\eta\pi)$ where $\pi$ is a uniformly random permutation of the labels $\{j\in[n]: j\in A\}$. This produces a valid coupling with the property that $ \dist(\sigma',\eta')=\dist(\sigma,\eta)$. Thus, hypothesis (i) of Theorem \ref{th:main} is satisfied with $\ell\equiv 1$. 

To check the positive curvature, we use a finer coupling. We first observe that 
the metric $\dist(\sigma,\eta)$ is generated by pairs of configurations that differ by exactly one swap, and thus using a telescopic decomposition, using the  so-called \emph{Gluing Lemma} (see, e.g., \cite[Lemma 7.6]{MR1964483}), we may restrict to a pair $(\sigma,\eta)$ such that $\sigma_\ell=\eta_\ell$, for all $\ell\neq i,j$, and $(\sigma_i,\sigma_j)=(\eta_j,\eta_i)$, for a fixed pair of distinct indexes $i,j\in[n]$. This reduction step is also known as the \emph{path coupling} lemma \cite{646111}. Given such a pair $(\sigma,\eta)$, the coupling of $(\sigma',\eta')$ is described as follows. Let $\pi$ be a uniformly random permutation of the labels $\{j\in[n]: j\in A\}$. If either $i$ or $j$ (or both) are in $\bA$ then we update as before by setting $(\sigma',\eta')=(\sigma\pi,\eta\pi)$. If instead 
$i,j$ are both in $A$, then we ensure coalescence by replacing $(\sigma,\eta)$ with $(\sigma\pi,\sigma\pi)$. This shows that 
\begin{align}
\sum_{A\subset [n]}\theta_A \,W_1\left(T_A(\sigma,\cdot),T_A(\eta,\cdot)\right) \le  
\sum_{A\subset [n]}\theta_A \,{\bf 1}_{\{A\not\supset\{i,j\}\}}\le 1-\theta_{\star\star},
\end{align}
where $\theta_{\star\star}:=\min_{1\le i<j\le n}\sum_{ A\supset\{i,j\}}\theta_A$. Thus, we have checked the validity of hypothesis (ii) of Theorem \ref{th:main} with $\kappa=\theta_{\star\star}$. This ends the proof of Theorem \ref{th:permutations}. 

\subsection{Contractive coupling for $\ell^n_p$--spheres}\label{sec:lp}
Here we present a
 contractive coupling towards the proof of Conjecture \ref{claim}.
 As in \eqref{eq:lp}, 
for any $p>0$, 
we take $\Omega=\bbS^{n-1}_p$, with the associated Borel $\sigma$-algebra, and write $\PP$ for the cone measure, i.e.\ the probability measure obtained as in \eqref{eq:gabe}. As usual, we write $Z:\Omega\mapsto\Omega$ for the identity map and take 
$T_Af := E[f|Z_{A^c}]$ for the conditional expectation of $f\in L^1$ given $Z_{A^c}$. We introduce the distance on $\Omega$ given by 
\begin{align}
\label{eq:distome}
\dist(x,y)= 
 \sum_{i=1}^n||x_i|^p - |y_i|^p| +  \sum_{i=1}^n{\bf 1}_{x_iy_i<0}  \,.
\end{align}
Under this metric, $T_A$ has nonnegative sectional curvature. More precisely,
\begin{lemma}\label{lem:Lemmad} 
For all $p>0$ for all $ x, y\in\Omega$,  and all $A\subset [n]$,  
\[
W_\infty(T_A(x,\cdot),T_A(y,\cdot)) = \sum_{i\in A^c}||x_i|^p -|y_i|^p|+\sum_{i\in A^c}{\bf 1}_{\{x_iy_i<0\}} +\left| \sum_{i\in A}(|x_i|^p-|y_i|^p)\right| \,.
\]
In particular, $W_\infty(T_A(x, \cdot), T_A(y, \cdot)) \le \dist(x, y)$. Moreover, the same identity applies to $W_1$.
\end{lemma}
\begin{proof}
Fix $x, y \in\Omega$ and consider any coupling $(X, Y )$ of $T_A(x, \cdot)$ and $T_A(y, \cdot)$. Since only the coordinates in $A$ are modified, we have almost-surely
\begin{align*}
\dist(X,Y) &= \sum_{i\in A^c}||x_i|^p -|y_i|^p|+\sum_{i\in A^c}{\bf 1}_{\{x_iy_i<0\}} + \sum_{i\in A}||X_i|^p-|Y_i|^p| 
+\sum_{i\in A^c}{\bf 1}_{\{x_iy_i<0\}}\\
& \ge \sum_{i\in A^c}||x_i|^p -|y_i|^p|+\sum_{i\in A^c}{\bf 1}_{\{x_iy_i<0\}} + \left|\sum_{i\in A}(|X_i|^p-|Y_i|^p)\right| 
\\
& =\sum_{i\in A^c}||x_i|^p -|y_i|^p|+\sum_{i\in A^c}{\bf 1}_{\{x_iy_i<0\}} + \left|\sum_{i\in A}(|x_i|^p-|y_i|^p)\right|\,,
\end{align*}
where the last line uses the fact that the norm is conserved $\|X_A\|_p = \|x_A\|_p$ and $\|Y_A\|_p = \|y_A\|_p$.
Now, the middle inequality is an equality as soon as $X_A$ is a positive multiple of $Y_A$. The latter fact is  almost-surely the case under the following coupling:
\[
(X_i,Y_i) = \begin{cases}
\left(\tfrac{\|x_A\|_p}{\|Z_A\|_p}\,Z_i\,,\,\tfrac{\|y_A\|_p}{\|Z_A\|_p}\,Z_i\right) & \;\text{ if }\,i\in A
\\
(x_i,y_i)& \;\text{ if }\, i\in A^c
\end{cases}
\]
where we recall that $Z$ has law $\PP$.
The fact that the above  is a valid coupling follows from the resampling property of the Dirichlet-type random variable with distribution given by \eqref{eq:gabe}. Thus we have found an optimal coupling, and the proof is complete.
\end{proof}
We now fix a probability vector $\theta=(\theta_A, A\subset [n])$, and we look for a constant $\kappa$  as
large as possible, such that for all $x, y \in\Omega$,
\begin{equation}\label{eq:kabe}
\sum_{A\subset[n]}\theta_AW_1 (T_A(x,\cdot), T_A(y,\cdot))\;\le\; (1-\kappa) \dist(x, y)\,. 
\end{equation}
The answer turns out to be remarkably explicit.
\begin{lemma}\label{lem:kabe}
The optimal constant in \eqref{eq:kabe} is exactly $\kappa=\theta_{\star\star}=\min_{i<j}\sum_{A\supset\{i,j\}}\theta_A$.
\end{lemma}
\begin{proof}
By the triangle inequality, and using the path coupling lemma \cite{646111}, the validity of the inequality  \eqref{eq:kabe}  can be checked separately on two complementary cases: the ``same sign" case where $x_iy_i \ge 0$ for all $i$, and the ``sign flip" case where $x, y$ differ at a single coordinate (its sign being the only difference). More precisely, the optimal constant that we are looking for is just the minimum of the two optimal constants obtained by restricting our attention to same-sign pairs, or to sign-flip pairs, respectively. The sign-flip case is actually easy: if $x, y$ differ at a single coordinate $i$, then $\dist(x, y) = 1$ and $W_1(T_A(x,\cdot), T_A(y,\cdot) = 1_{i\in A^c}$, so that
\[
\sum_{A\subset[n]}\theta_AW_1 (T_A(x,\cdot), T_A(y,\cdot))\;\le\; (1-\theta_\star) \dist(x, y)\,,
\]
with equality when $i$ realizes the minimum in the definition of $\theta_\star=\min_{i}\sum_{A\ni i}\theta_A$. Let us now turn to the ``same sign" case: given a pair $(x,y) \in\Omega^2$, with $x\neq y$ and $x_iy_i \ge 0$ for all $i$, we introduce two probability vectors $\mu,\nu$ on $[n]$, with disjoint supports, defined by
\[
\mu_i:=\frac{2(|x_i|^p -|y_i|^p)_+}{\sum_{i=1}^n||x_i|^p - |y_i|^p| }
\,,\qquad \nu_i:=\frac{2(|y_i|^p -|x_i|^p)_+}{\sum_{i=1}^n||x_i|^p - |y_i|^p| }
\,,
\]
where $(a)_+=\max\{a,0\}$ denotes the positive part of $a\in\R$. 
Using Lemma \ref{lem:Lemmad}, we then find
\begin{align*}
\frac{W_1 (T_A(x,\cdot), T_A(y,\cdot))}{\dist(x,y)} &= 1- \frac{\sum_{i\in A}||x_i|^p-|y_i|^p|}{\sum_{i=1}^n||x_i|^p - |y_i|^p| }
+ \frac{\left|\sum_{i\in A}(|x_i|^p-|y_i|^p)\right| 
}{\sum_{i=1}^n||x_i|^p - |y_i|^p| }
 \\ & = 1- \frac{\mu(A)+\nu(A)-|\mu(A)-\nu(A)|}{2}
\\ & = 1- \min\{\mu(A),\nu(A)\}\,.
\end{align*}
Consequently,
\begin{align*}
\frac{\sum_{A\subset[n]}\theta_AW_1 (T_A(x,\cdot), T_A(y,\cdot))}{\dist(x,y)} &\le 
1- \sum_{A\subset[n]}\theta_A\mu(A)\nu(A)\\
& = 1- \sum_{A\subset[n]}\theta_A\sum_{i\in A,j\in A}\mu_i\nu_i\le 1- \theta_{\star\star}.
\end{align*}
Moreover, the two inequalities appearing in this computation are both equalities if $x$ and $y$ are the $i$-th and $j$-th vectors of the canonical basis of $\R^n$, with $(i,j)$ being any pair that realizes the minimum in the definition of $\theta_{\star\star}$. Thus, the optimal constant in the ``same sign" optimization problem is precisely $\theta_{\star\star}$. To conclude, it remains to note the simple fact that $\theta_{\star\star}\le \theta_{\star}$.
\end{proof}
As we mentioned in Section \ref{sec:appl}, Lemma \ref{lem:kabe} would be sufficient to prove Conjecture \ref{claim} if we could solve the technical problem of replacing Assumption (iii) in Theorem \ref{th:main} by a suitable regularizing procedure as outlined in Remark \ref{rk:regul}. 

\subsection{Down-Up walk on uniform $n$-sets: proof of Theorem \ref{th:downup}}
\label{sec:downup}
To prove the theorem we apply our general result from Theorem \ref{th:main}. In fact, it will be sufficient to apply it for each $k$ separately, and then obtain \eqref{eq:downup} by summing over $k$. Given $X,Y\in\Omega$ we write $|X\cap Y|$ for the number of elements in $X\cap Y$, and define 
 \[
\dist(X,Y)= n - |X\cap Y|\,,
\]
for the number of discrepancies in the two sets. Notice that this defines a distance in $\Omega$. Indeed, we can identify $X\in\Omega$ with the function $\eta_X:\bS\mapsto\{0,1\}$ such that $\eta_X(x) = {\bf 1}_{x\in X}$ and in this representation one has
$\dist(X,Y)= \frac12\sum_{x\in\bS}{\bf 1}_{\eta_X(x) \neq \eta_{Y}(x) }$. Therefore, Theorem \ref{th:downup} is an immediate consequence of Theorem \ref{th:main} and the following lemma.

\begin{lemma}\label{lem:partialperm}
For any $k=1,\dots,n$, for any $X,Y\in\Omega$, 
$W_\infty(T_k(X,\cdot),T_k(Y,\cdot))\leq \dist(X,Y)$, and 
\begin{eqnarray}
\label{eq:parperm}
W_1(T_k(X,\cdot),T_k(Y,\cdot))
\leq \left(1- \frac{k}{n}\right)\left(1- \frac{k}{N-(n-k)}\right)\dist(X,Y)\,.
\end{eqnarray}
\end{lemma}
\begin{proof}
We want to estimate 
$\EE[\dist(X',Y')]$, where $(X',Y')$ is  
a suitable coupling of the laws $T_k(X,\cdot)$ and $T_k(Y,\cdot)$ respectively. By the path coupling lemma \cite{646111} we may restrict to the case  where $X,Y\in\Omega$ is an arbitrary pair such that $\dist(X,Y)=1$. Thus, there exist $x,y\in\bS$ such that $x\neq y$, $x\in X\cap Y^c$, $y\in Y\cap X^c$ and $X\cap Y=X\setminus\{x\}=Y\setminus\{y\}$. Let $X_{-k}$ denote a  uniformly random subset of $X$ with cardinality $n-k$. There are $\binom{n}{k}$ possible choices, $\binom{n-1}{k-1}$ of which  do not contain $x$.
If  $x\notin  X_{-k}$ then we may set $Y_{-k}=X_{-k}$ and thus define $(X',Y')=(X',X')$, where $X'$ is a  uniformly random supset of $X_{-k}$ with cardinality $n$. If instead $x\in  X_{-k}$, we may take $Y_{-k}=(X_{-k}\setminus\{x\})\cup\{y\}$ and we construct $(X',Y')$ as follows. Let $X'=X_{-k}\cup X^k$, where $X^k$ is a  uniformly random subset of $\bS\setminus X_{-k}$ with cardinality $k$. There are $\binom{N-(n-k)}{k}$ possible choices, $\binom{N-(n-k)-1}{k-1}$ of which contain $y$. If $y\notin X^k$, then we set $Y'=(X'\setminus\{x\})\cup\{y\}$. If instead $y\in X^k$, then we set $Y'=X'$. 

It is not hard to check that the above defines a valid coupling of $(X',Y')$, such that $\dist(X',Y')=\dist(X,Y)=1$ with probability 
\[
\PP(y\notin X^k) =\left(1- \frac{k}{n}\right)\left(1- \frac{k}{N-(n-k)}\right)\,,
\]
 while with the remaining probability one has $\dist(X',Y')=0$. In particular, it follows that 
 \[
 W_\infty(T_k(X,\cdot),T_k(Y,\cdot))\leq \dist(X,Y)\,,
 \] 
 and
\begin{eqnarray}
\EE[\dist(X',Y')]\leq \left(1- \frac{k}{n}\right)\left(1- \frac{k}{N-(n-k)}\right).
\end{eqnarray}
\end{proof}

\subsection{Langevin diffusion: proof of Theorem \ref{th:BE}}\label{sec:langevin}
The proof is an immediate application of Theorem \ref{th:main}. Given $x,y\in\R^n$, we may couple two Langevin diffusions $X$ and $Y$ starting from $X_0=x$ and $Y_0=y$ by using the same Brownian motion $B=(B_t)_{t\ge 0}$, so that 
\begin{eqnarray*}
\mathrm d\|X_t-Y_t\|^2 & =& -2\langle X_t-Y_t,\nabla V(X_t)-\nabla V(Y_t)\rangle\dd t\\
& \le & -2\rho \|X_t-
Y_t\|^2\dd t,
\end{eqnarray*}
where the second line uses the $\rho-$convexity  (\ref{assume:convex}). Thus, the Euclidean distance $t\mapsto \|X_t-Y_t\|$ decays at rate at least $\rho$, almost-surely. This shows that for all $t\ge 0$ and all $x,y\in\R^n$,
\begin{eqnarray}
W_\infty\left(P_t(x,\cdot),P_t(y,\cdot)\right) & \le & e^{-\rho t}\|x-y\|. 
\end{eqnarray}
We may now apply the case $M=1$ of Theorem \ref{th:main} to the operator $T=P_t$, for a fixed $t>0$. Since $P_t^\star=P_t$ and $W_1\le W_\infty$, Assumptions (i)-(ii) hold with $\ell=e^{-\rho t}$ and $\kappa=1-e^{-2\rho t}$. Moreover, the regularity property (iii) classically holds in this setting (see, e.g., \cite[Th.\ 4.7.2]{bakry2014analysis}). The desired conclusion readily follows.

\bibliographystyle{plain}
\bibliography{entropy}

\begin{thebibliography}{10}

\bibitem{ahlswede1976spreading}
Rudolf Ahlswede and Peter G{\'a}cs.
\newblock Spreading of sets in product spaces and hypercontraction of the
  {M}arkov operator.
\newblock {\em The Annals of Probability}, pages 925--939, 1976.

\bibitem{anari2022entropic}
Nima Anari, Vishesh Jain, Frederic Koehler, Huy~Tuan Pham, and Thuy-Duong
  Vuong.
\newblock Entropic independence: optimal mixing of down-up random walks.
\newblock In {\em Proceedings of the 54th Annual ACM SIGACT Symposium on Theory
  of Computing}, pages 1418--1430, 2022.

\bibitem{bakry2014analysis}
Dominique Bakry, Ivan Gentil, Michel Ledoux, et~al.
\newblock {\em Analysis and geometry of Markov diffusion operators}, volume
  103.
\newblock Springer, 2014.

\bibitem{barthe2006entropy}
F~Barthe, Dario Cordero-Erausquin, and B~Maurey.
\newblock Entropy of spherical marginals and related inequalities.
\newblock {\em Journal de math{\'e}matiques pures et appliqu{\'e}es},
  86(2):89--99, 2006.

\bibitem{barthe1998reverse}
Franck Barthe.
\newblock On a reverse form of the {B}rascamp-{L}ieb inequality.
\newblock {\em Inventiones mathematicae}, 134(2):335--361, 1998.

\bibitem{barthe2011correlation}
Franck Barthe, Dario Cordero-Erausquin, Michel Ledoux, and Bernard Maurey.
\newblock Correlation and {B}rascamp--{L}ieb inequalities for markov
  semigroups.
\newblock {\em International Mathematics Research Notices},
  2011(10):2177--2216, 2011.

\bibitem{barthe2005probabilistic}
Franck Barthe, Olivier Gu{\'e}don, Shahar Mendelson, and Assaf Naor.
\newblock A probabilistic approach to the geometry of the ? pn-ball.
\newblock 2005.

\bibitem{bennett2008brascamp}
Jonathan Bennett, Anthony Carbery, Michael Christ, and Terence Tao.
\newblock The {B}rascamp--{L}ieb inequalities: finiteness, structure and
  extremals.
\newblock {\em Geometric and Functional Analysis}, 17(5):1343--1415, 2008.

\bibitem{MR4415182}
Antonio Blanca, Pietro Caputo, Zongchen Chen, Daniel Parisi, Daniel
  \v{S}tefankovi\v{c}, and Eric Vigoda.
\newblock On mixing of {M}arkov chains: coupling, spectral independence, and
  entropy factorization.
\newblock In {\em Proceedings of the 2022 {A}nnual {ACM}-{SIAM} {S}ymposium on
  {D}iscrete {A}lgorithms ({SODA})}, pages 3670--3692. SIAM, Philadelphia, PA,
  2022.

\bibitem{MR3185193}
St\'{e}phane Boucheron, G\'{a}bor Lugosi, and Pascal Massart.
\newblock {\em Concentration inequalities}.
\newblock Oxford University Press, Oxford, 2013.
\newblock A nonasymptotic theory of independence, With a foreword by Michel
  Ledoux.

\bibitem{brascamp1976best}
Herm~Jan Brascamp and Elliott~H Lieb.
\newblock Best constants in young's inequality, its converse, and its
  generalization to more than three functions.
\newblock {\em Advances in Mathematics}, 20(2):151--173, 1976.

\bibitem{BristielC}
Alexandre Bristiel and Pietro Caputo.
\newblock Entropy inequalities for random walks and permutations.
\newblock In {\em Annales de l'Institut Henri Poincare (B) Probabilites et
  statistiques}, volume~60, pages 54--81. Institut Henri Poincar{\'e}, 2024.

\bibitem{646111}
R.~Bubley and M.~Dyer.
\newblock Path coupling: A technique for proving rapid mixing in markov chains.
\newblock In {\em Proceedings 38th Annual Symposium on Foundations of Computer
  Science}, pages 223--231, 1997.

\bibitem{entropy}
Pietro Caputo.
\newblock Lecture notes on entropy and {M}arkov chains.
\newblock {\em Available at:
  http://www.mat.uniroma3.it/users/caputo/entropy.pdf}, 2022.

\bibitem{MR3434252}
Pietro Caputo, Georg Menz, and Prasad Tetali.
\newblock Approximate tensorization of entropy at high temperature.
\newblock {\em Ann. Fac. Sci. Toulouse Math. (6)}, 24(4):691--716, 2015.

\bibitem{caputo2024entropy}
Pietro Caputo, Florentin M{\"u}nch, and Justin Salez.
\newblock Entropy and curvature: beyond the peres-tetali conjecture, 2024.

\bibitem{MR4334247}
Pietro Caputo and Daniel Parisi.
\newblock Block factorization of the relative entropy via spatial mixing.
\newblock {\em Comm. Math. Phys.}, 388(2):793--818, 2021.

\bibitem{carlen2009subadditivity}
Eric~A Carlen and Dario Cordero-Erausquin.
\newblock Subadditivity of the entropy and its relation to {B}rascamp--{L}ieb
  type inequalities.
\newblock {\em Geometric and Functional Analysis}, 19(2):373--405, 2009.

\bibitem{carlen2004sharp}
Eric~A Carlen, Elliott~H Lieb, and Michael Loss.
\newblock A sharp analog of young’s inequality on sn and related entropy
  inequalities.
\newblock {\em The Journal of Geometric Analysis}, 14:487--520, 2004.

\bibitem{MR3413687}
Filippo Cesi.
\newblock A few remarks on the octopus inequality and {A}ldous' spectral gap
  conjecture.
\newblock {\em Comm. Algebra}, 44(1):279--302, 2016.

\bibitem{courtade2024rigid}
Thomas~A Courtade.
\newblock Rigid characterizations of probability measures through independence,
  with applications.
\newblock {\em arXiv preprint arXiv:2403.06615}, 2024.

\bibitem{cryan2021modified}
Mary Cryan, Heng Guo, and Giorgos Mousa.
\newblock Modified log-sobolev inequalities for strongly log-concave
  distributions.
\newblock {\em The Annals of Probability}, pages 506--525, 2021.

\bibitem{dai2002entropy}
Paolo Dai~Pra, Anna~Maria Paganoni, and Gustavo Posta.
\newblock Entropy inequalities for unbounded spin systems.
\newblock {\em The Annals of Probability}, 30(4):1959--1976, 2002.

\bibitem{friedli2017statistical}
Sacha Friedli and Yvan Velenik.
\newblock {\em Statistical mechanics of lattice systems: a concrete
  mathematical introduction}.
\newblock Cambridge University Press, 2017.

\bibitem{ganguly2023cutoff}
Shirshendu Ganguly and Reza Gheissari.
\newblock Cutoff for the {G}lauber dynamics of the lattice free field.
\newblock {\em Probability and Mathematical Physics}, 4(2):433--475, 2023.

\bibitem{lehec2014short}
Joseph Lehec.
\newblock Short probabilistic proof of the {B}rascamp-{L}ieb and {B}arthe
  theorems.
\newblock {\em Canadian Mathematical Bulletin}, 57(3):585--597, 2014.

\bibitem{lieb1990gaussian}
Elliott~H Lieb.
\newblock Gaussian kernels have only gaussian maximizers.
\newblock {\em Inventiones mathematicae}, 102:179--208, 1990.

\bibitem{liu2017information}
Jingbo Liu, Thomas~A Courtade, Paul Cuff, and Sergio Verdu.
\newblock Information-theoretic perspectives on {B}rascamp-{L}ieb inequality
  and its reverse.
\newblock {\em arXiv preprint arXiv:1702.06260}, 2017.

\bibitem{Lu-Yau}
Sheng~Lin Lu and Horng-Tzer Yau.
\newblock Spectral gap and logarithmic {S}obolev inequality for {K}awasaki and
  {G}lauber dynamics.
\newblock {\em Comm. Math. Phys.}, 156(2):399--433, 1993.

\bibitem{MR2683430}
Mokshay Madiman and Prasad Tetali.
\newblock Information inequalities for joint distributions, with
  interpretations and applications.
\newblock {\em IEEE Trans. Inform. Theory}, 56(6):2699--2713, 2010.

\bibitem{MR1746301}
Fabio Martinelli.
\newblock Lectures on {G}lauber dynamics for discrete spin models.
\newblock In {\em Lectures on probability theory and statistics
  ({S}aint-{F}lour, 1997)}, volume 1717 of {\em Lecture Notes in Math.}, pages
  93--191. Springer, Berlin, 1999.

\bibitem{MR2995699}
Katalin Marton.
\newblock An inequality for relative entropy and logarithmic {S}obolev
  inequalities in {E}uclidean spaces.
\newblock {\em J. Funct. Anal.}, 264(1):34--61, 2013.

\bibitem{MR4015662}
Katalin Marton.
\newblock Logarithmic {S}obolev inequalities in discrete product spaces.
\newblock {\em Combin. Probab. Comput.}, 28(6):919--935, 2019.

\bibitem{munch2023ollivier}
Florentin M{\"u}nch.
\newblock Ollivier curvature, isoperimetry, concentration, and log-{S}obolev
  inequalitiy, 2023.

\bibitem{ollivier2009ricci}
Yann Ollivier.
\newblock Ricci curvature of {M}arkov chains on metric spaces.
\newblock {\em Journal of Functional Analysis}, 256(3):810--864, 2009.

\bibitem{pedrotti2023contractive}
Francesco Pedrotti.
\newblock Contractive coupling rates and curvature lower bounds for markov
  chains, 2023.

\bibitem{SZ92a}
D.~W. Stroock and B.~Zegarlinski.
\newblock The equivalence of the logarithmic {S}obolev inequality and the
  {D}obrushin-{S}hlosman mixing condition.
\newblock {\em Communications in Mathematical Physics}, 144(2):303--323, 1992.

\bibitem{van2014probability}
Ramon Van~Handel.
\newblock Probability in high dimension.
\newblock {\em Lecture Notes (Princeton University)}, 2014.

\bibitem{villani2003cercignani}
C{\'e}dric Villani.
\newblock Cercignani's conjecture is sometimes true and always almost true.
\newblock {\em Communications in mathematical physics}, 234:455--490, 2003.

\bibitem{MR1964483}
C\'{e}dric Villani.
\newblock {\em Topics in optimal transportation}, volume~58 of {\em Graduate
  Studies in Mathematics}.
\newblock American Mathematical Society, Providence, RI, 2003.

\bibitem{MR1155402}
David Williams.
\newblock {\em Probability with martingales}.
\newblock Cambridge Mathematical Textbooks. Cambridge University Press,
  Cambridge, 1991.

\end{thebibliography}
\end{document}